\documentclass[11pt]{amsart}
\usepackage{epsfig}
\usepackage{amscd}
\usepackage[mathscr]{eucal}
\usepackage{amssymb}
\usepackage{amsxtra}
\usepackage{amsmath}
\usepackage[all]{xy}
\usepackage{color}

\newtheorem{thm}{Theorem} 

\newtheorem{prop}{Proposition}[section]

\newtheorem{cor}[prop] {Corollary} 

\newtheorem{rem}[prop]{Remark} 

\newtheorem{lem}[prop]{Lemma} 

\newtheorem{que}[prop]{Question} 

\newtheorem{defn}[prop]{Definition} 

\newtheorem{exmp}[prop]{Example}

\begin{document}
\title[$C^0$-limits and spectral invariants]{$C^0$-limits of Hamiltonian paths and the Oh-Schwarz spectral invariants }

%\date{\today}
\author[Sobhan Seyfaddini]{Sobhan Seyfaddini}

\address{University of California Berkeley\\ Berkeley, CA 94720 \\USA}
\email{sobhan@math.berkeley.edu}

\begin{abstract}

{\noindent In this article we study the behavior of the Oh-Schwarz spectral invariants under $C^0$-small perturbations of the Hamiltonian flow.  We show that if two Hamiltonians $G, H$ vanish on a small ball and if their flows are sufficiently $C^0$-close then  $$|\rho(G;a) - \rho(H;a)| \leq C \, d_{C^0}^{path}(\phi^t_G, \phi^t_H).$$

    Using the above result, we prove that if $\phi$ is a sufficiently $C^0$-small Hamiltonian diffeomorphism on a surface of genus $g$ then
   $$\Vert \phi \Vert_{\gamma} \leq C \, (d_{C^0}(Id, \phi))^{2^{-2g-1}}$$
   hence establishing $C^0$-continuity of the spectral norm on surfaces.
   
   We also present applications of the above results to the theory of Calabi quasimorphisms and improve a result of Entov, Polterovich and Py \cite{EPP}.  In the final section of the paper we use our results to answer a question of Y.-G. Oh about spectral Hamiltonian homeomorphisms.
   
}
\end{abstract}

\maketitle

\section{{\bf Introduction}}
  The 1-periodic orbits of a non-degenerate Hamiltonian $H \in C^{\infty}([0,1] \times M)$, which are the critical points of the action functional associated to $H$, generate the Floer chain complex, $CF_*(H)$, of the Hamiltonian $H$; see \cite{Sa} for a survey of Floer theory on a large class of symplectic manifolds.  The important fact that the homology of this complex, denoted by $HF_*(H)$, coincides with the quantum cohomology of $M$, $QH^*(M) = H^*(M) \otimes \Lambda$, is one of the most spectacular achievements of Floer theory.  Using the isomorphism between $QH^*(M)$ and $HF_*(H)$, one defines the spectral invariant associated to $H$ and $a \in QH^*(M) \setminus \{0\}$, denoted by $\rho(H;a)$, to be the minimum action required to represent the quantum cohomology class $a$ in $HF_*(H)$.  Spectral invariants were introduced in the symplectically aspherical case by Schwartz in \cite{Sc} and were extended to general closed symplectic manifolds by Oh \cite{Oh1}.  We will briefly review this construction in section \ref{Spectral Invariants}; see \cite{Oh2} for a full survey.  One interesting application of the spectral invariant $\rho(\cdot ;1)$ is the construction of a bi-invariant norm, called the spectral norm, on $Ham(M)$.  For other interesting applications of spectral invariants we refer the reader to \cite{EP1, EP2, EP3, EP4, Oh2, Oh3}.  
  
  In this paper we study the behavior of the spectral invariants of a Hamiltonian $H$, and the spectral norm of a Hamiltonian diffeomorphism $\psi$ under $C^0$-small perturbations of the flow of $H$ and the diffeomorphism $\psi$, respectively.  In section \ref{Spectral Invariants} we provide a brief introduction to spectral invariants.  Section \ref{Main Results} contains the statements of the main results of the paper.  In section \ref{The Lipschitz Estimate}, we prove the first main theorem of this note: under the assumptions that $H, G$ vanish on a small ball and that $\phi^t_H, \phi^t_G$ are sufficiently $C^0$-close , we will bound $|\rho(H;1) - \rho(G;1)|$ by a multiple of the $C^0$-distance of $\phi^t_{H}$ and $\phi^t_G$.  In section \ref{Surfaces} we compare the spectral norm and the $C^0$-norm on surfaces and show that the spectral norm is continuous with respect to the $C^0$-topology on $Ham$.  In section \ref{Application to the theory of Calabi quasimorphisms}, after briefly introducing the theory of Calabi quasimorphisms, we present applications of our results to this theory.  In the final section of this paper, section \ref{Application to Spectral Hamiltonian Homeomorphisms}, we recall Oh's definition of spectral Hamiltonian homeomorphisms and we answer a question raised by him on this subject.
  
 {\bf Acknowledgments: } This paper grew out of an attempt to better understand an idea of Yong-Geun Oh.  I thank him for his interest in this project.  I would like to express my gratitude to my advisor, Alan Weinstein, for his time, support, and insightful guidance.  I feel heavily indebted to Denis Auroux for generously sharing his time and ideas with me.  I'm grateful to Leonid Polterovich for a very helpful suggestion and for drawing my attention to the results in \cite{EPP}, which play an important role in this paper.  I'm also thankful to Claude Viterbo for pointing out the connection between the results of this paper and an ongoing project of his, and for making a preliminary version of his work available to me.  For interesting comments and stimulating conversations I'd like to thank Lev Buhovsky, Yakov Eliashberg, Viktor Ginzburg, Michael Hutchings and Beno\^it Jubin.

\subsection{Preliminaries on spectral invariants} \label{Spectral Invariants}  
In this section we set our notation and briefly review the necessary background on spectral invariants and Hofer geometry.  For further details on these subjects, we refer the interested reader to \cite{Oh2}, \cite{MS2}, \cite{HZ} and \cite{P}.
    
    Let $(M, \omega)$ denote a closed and connected symplectic manifold.  Any smooth Hamiltonian $H:[0,1] \times M \rightarrow \mathbb{R}$ induces a Hamiltonian flow $\phi^t_H : M \rightarrow M \text{ } (0\leq t \leq 1)$, by integrating the unique time-dependent vector field $X_H$ satisfying $dH_t = \iota_{X_H}\omega$, where $H_t(x) = H(t,x)$.  We denote the space of Hamiltonian flows by $PHam(M, \omega)$.  A Hamiltonian diffeomorphism is by definition any diffeomorphism obtained as the time-1 map of a Hamiltonian flow.  We denote by $Ham(M, \omega)$ the set of all Hamiltonian diffeomorphisms.  We will eliminate the symplectic form $\omega$ from the above notations unless there is a possibility of confusion.
  
   Define 
   $$\Gamma:= \frac{\pi_2(M)}{ker(c_1) \cap ker([\omega])}.$$
   The Novikov ring of $(M, \omega)$ is defined to be
   $$ \displaystyle \Lambda = \{ \sum_{A\in \Gamma}{a_A A }: a_A \in \mathbb{Q}, (\forall C \in \mathbb{R}) (|\{A: a_A \neq 0, \int_{A} {\omega} < C\}| < \infty)\}.$$  
   Let $\Omega_0(M)$ denote the space of contractible loops in $M$.  $\Gamma$ forms the group of deck transformations of a covering $\tilde{\Omega}_0(M) \rightarrow \Omega_0(M)$ called the Novikov covering of $\Omega_0(M)$ which can be described as follows:
   $$\tilde{\Omega}_0(M) = \frac{ \{ [z,u]: z \in \Omega_0(M) , u: D^2 \rightarrow M , u|_{\partial D^2} = z \}}{[z,u] = [z', u'] \text { if } z=z' \text{ and } \bar{u} \# u' = 0 \text{ in } \Gamma},$$
   where $\bar{u} \# u'$ denotes the sphere obtained by gluing $u$ and $u'$ along their common boundary with the orientation on $u$ reversed.
   The action functional, associated to a Hamiltonian $H \in C^{\infty}([0,1]\times M)$, is the map $\mathcal{A}_H:\tilde{\Omega}_0(M) \rightarrow \mathbb{R}$ given by
   $$\mathcal{A}_H([z,u]) =  - \int H(t,z(t))dt \text{ }- \int_{u} \omega.$$
   $Crit(\mathcal{A}_H) = \{ [z,u] : \text{ z is a 1-periodic orbit of } X_H \}$ denotes the set of critical points of $\mathcal{A}_H$.  The action spectrum of $H$ is defined to be the set of critical values of the action functional, i.e., $Spec(H) = \mathcal{A}_H(Crit(\mathcal{A}_H))$.  $Spec(H)$ is a measure zero subset of $\mathbb{R}$.
   
  We say that a Hamiltonian $H$ is non-degenerate if the graph of $\phi^1_H$ intersects the diagonal in $M \times M$ transversally.  The Floer chain complex of (non-degenerate) $H$, $CF_*(H)$, is generated as a module over $\Lambda$ by $Crit(\mathcal{A}_H)$.  The boundary map of this complex is obtained, formally, by counting isolated negative gradient flow lines of $\mathcal{A}_H$.  The homology of this complex, $HF_*(H)$, is naturally isomorphic to $QH^*(M) = H^*(M) \otimes \Lambda$, the quantum cohomology of $M$.  We denote this natural isomorphism by 
   $\Phi : QH^*(M) \rightarrow HF_*(H)$.  
   
   Given $\displaystyle \alpha = \sum_{[z,u] \in Crit(\mathcal{A}_H)}{a_{[z,u]}[z,u]} \in CF_*(H)$ we define the action level of $\alpha$ by
   $$\lambda_H(\alpha) = max\{\mathcal{A}_H([z,u]): a_{[z,u]} \neq  0 \}.$$
   Finally, given a non-zero quantum cohomology class $a$, we define the spectral invariant associated to $H$ and $a$ by
   $$\rho(H;a) = inf \{ \lambda_H(\alpha): [\alpha] = \Phi(a)\},$$
   where $[\alpha]$ denotes the Floer homology class of $\alpha$.  It was shown in \cite{Oh1} that $\rho(H;a)$ is well defined, i.e., it is independent of the auxiliary data (almost complex structure) used to define it and $\rho(H;a) \neq -\infty$.
     
   Thus far we have defined $\rho(H;a)$ for non-degenerate $H$.  Define the $L^{(1,\infty)}$ (or Hofer) norm of a Hamiltonian $K$ by \[ \Vert K  \Vert_{(1,\infty)} = \int_0^1 (\max_x K(t,x) - \min_x K(t,x))\,dt.\]
   The spectral invariants of two non-degenerate Hamiltonians $H_1$, $H_2$ satisfy the following estimate
   $$|\rho(H_1;a) - \rho(H_2;a)| \leq \Vert H_1 - H_2  \Vert_{(1,\infty)}.$$
   This estimate allows us to extend $\rho(\cdot;a)$ continuously to all smooth (in fact continuous) Hamiltonians.  
   
   We will now list, without proof, some properties of $\rho$ which will be used later on.  Recall that if $H, \text{ }G$ are two smooth Hamiltonians with flows $\phi^t_H, \text{ } \phi^t_G$, then the composition of the two flows, $\phi^t_H \circ \phi^t_G$, is a Hamiltonian flow generated by the Hamiltonian $H\#G(t,x) = H(t,x) + G(t, (\phi^t_H)^{-1}(x))$.  One can also check that the Hamiltonian $\bar{H}(t,x) = -H(t, \phi^t_H(x))$ generates the inverse flow $(\phi^{t}_H)^{-1}$.  A Hamiltonian $H$ is said to be normalized if $\displaystyle \int_{M}{H_t \omega^n} = 0$ for each $t\in [0,1]$.  
   
   \begin{prop} \label{Properties of Spectral Invariants}(\cite{Oh1}, \cite{Oh2})\\
    The function $\rho: C^{\infty}([0,1] \times M) \times (QH^*(M) \setminus {0}) \rightarrow \mathbb{R}$ has the following properties:
    \begin{enumerate}
   
    \item  \label{Shift}If $r : [0,1] \rightarrow \mathbb{R}$ is smooth then $$\rho(H+r;a) = \rho(H;a) - \int_{0}^{1}{r(t) dt}.$$
    \item \label{Normalization}(Normalization) $\rho(0;1) =0$.
    \item \label{Symplectic Invariance}(Symplectic Invariance) $\rho(\eta^*H; \eta^*a) = \rho(H;a)$ for any symplectomorphism $\eta$.
    \item \label{Triangle Inequality}(Triangle Inequality) $\rho(H\#G; a*b) \leq \rho(H;a) + \rho(G;b)$ where $*$ denotes the quantum product in $QH^*(M)$.
    \item \label{Continuity}($L^{(1,\infty)}-continuity$) 
    
             $\displaystyle - \int_{0}^{1} \max_{ M} (H_t - G_t) dt \leq  \rho(H;a) - \rho(G;a)
             \leq - \int_{0}^{1} \min_{M} (H_t - G_t) dt$, and in particular
             
             $\displaystyle - \int_{0}^{1} \max_{M} H_t \, dt  \leq  \rho(H;1)\leq - \int_{0}^{1} \min_{M} H_t \, dt $.
    \item \label{Spectrality}(Non-degenerate Spectrality) $\rho(H;a) \in Spec(H)$ for non-degenerate $H$.
    \item \label{Homotopy Invariance}(Homotopy Invariance) Assume that $\phi^1_H = \phi^1_G$ and that the Hamiltonian paths $\phi^t_H$ and $\phi^t_G$ are homotopic rel. endpoints.  If $H$ and $G$ are normalized, then $\rho(H;a) = \rho(G;a)$.
    \end{enumerate} 
   
   \end{prop}
   
   One can also define spectral invariants for Hamiltonian paths.  Given $\phi^t \in PHam(M)$ we take $H$ to be the unique normalized Hamiltonian generating $\phi^t$ and define $\rho(\phi^t;a) = \rho(H;a)$.  Note that the homotopy invariance property implies that $\rho(\cdot;a)$ descends to the universal cover of $Ham(M)$.  
   
   We define the spectral length function $\gamma : C^{\infty}([0,1] \times M) \rightarrow \mathbb{R}$ by 
   $$\gamma(H) = \rho(H;1) + \rho(\bar{H};1).$$
   Note that the triangle inequality implies that $\gamma$ is always non-negative.  Also, if $H$ and $G$ differ by a function of time, then part \ref{Shift} of Proposition \ref{Properties of Spectral Invariants} implies that $\gamma(H) = \gamma(G)$.  Hence, $\gamma$ is well defined on $PHam(M)$ and by the homotopy invariance property it in fact descends to the universal cover of $Ham(M)$.  By property (\ref{Continuity}) of spectral invariants we have $\gamma(H) \leq \Vert H \Vert_{(1, \infty)}$.
   
   Define the spectral norm of a Hamiltonian diffeomorphism $\psi$ by
   $$\Vert \psi \Vert_{\gamma} = \inf \{ \gamma(H): \psi = \phi^1_H \}.$$
   $\Vert \cdot \Vert_{\gamma}$ induces a non-degenerate norm on $Ham(M)$, see \cite{Oh3, U1} for a proof.  The spectral distance of two Hamiltonian diffeomorphisms is given by: $$d_{sp}(\phi, \psi) = \Vert \phi^{-1} \psi \Vert_{\gamma}.$$
   
   {\bf Warning:}  Some authors use the notation $\gamma(\cdot)$ for the spectral norm.  Please note that we are using that notation for a different purpose here.  \\
   
   Recall that the Hofer norm of a Hamiltonian diffeomorphism $\psi \in Ham(M)$, introduced in \cite{H1}, is given by the following expression:
  $$\Vert \psi \Vert_{Hofer} = \inf \{ \Vert H \Vert_{(1,\infty)}: \psi = \phi^1_H \}.$$
  Non-degeneracy of the above norm was established in \cite{H1} on $\mathbb{R}^{2n}$ and in \cite{LM} on general symplectic manifolds.  The Hofer distance of two Hamiltonian diffeomorphisms is given by: $$d_{Hofer}(\phi, \psi) = \Vert \phi^{-1} \psi \Vert_{Hofer}.$$

\subsection{Main Results} \label{Main Results}
   Throughout this paper, we equip $M$ with a distance $d$ induced by any Riemannian metric.  The $C^0$-topology on $Diff(M)$, the space of diffeomorphisms of $M$, is the topology induced by the distance $ \displaystyle  d_{C^0}(\phi, \psi):= \max_{x}d(x, \phi^{-1}\psi(x))$.  Similarly, for paths of diffeomorphisms $\phi^t, \psi^t$ ($t\in [0,1]$) we define their $C^0$-distance by the expression $\displaystyle d_{C^0}^{path}(\phi^t, \psi^t):= \max_{t,x} d(x, (\phi^t)^{-1}\psi^t(x)).$  Unless otherwise mentioned, we assume that both $Ham$ and $PHam$ are equipped with the above $C^0$ distances.
   
\subsubsection{{\bf A locally Lipschitz estimate for spectral invariants}}
   
   Recall that for a Hamiltonian path $\phi^t_H$, generated by normalized $H$, $\rho(\phi^t_H;a)$ is defined to be $\rho(H;a)$.  One may ask if the map $\rho(\cdot;a): PHam(M) \rightarrow \mathbb{R}$ is $C^0$-continuous.  The answer to this question turns out to be no.  In Example \ref{Discontinuity Example}, we will construct a sequence of normalized Hamiltonians $H_k$ such that the flows of these Hamiltonians $C^0$-converge to the identity, but the spectral invariants $\rho(H_k;1)$ and $\rho(\bar{H_k};1)$ converge to $1$ and $-1$, respectively.  However, as the first main theorem of this note demonstrates, it seems that the culprit here is the requirement that the generating Hamiltonians be normalized.
   
   Let $B$ be an open ball in $(M, \omega)$ and denote by $C^{\infty}_c([0,1] \times (M \setminus B))$ the set of smooth functions vanishing on $B$.
  
  \begin{thm} ($C^0$-Spectral Estimate) \label{Lipschitz Estimate}
      Suppose $H, G \in C^{\infty}_c([0,1] \times (M \setminus B))$.   There exist constants $C, \delta > 0$ (depending on $B$) such that if $d_{C^0}^{path}(\phi^t_G, \phi^t_H) < \delta$, then 
   
       $$|\rho(G;a) - \rho(H;a)| \leq C \, d_{C^0}^{path}(\phi^t_G, \phi^t_H).$$
      
  \end{thm}
  
  In Example \ref{Sharpness of the Lipschitz Estimate} we prove that the above estimate is sharp in the sense that a locally Lipschitz estimate is optimal. 
   The proof of Theorem \ref{Lipschitz Estimate} relies on the following variation of the concept of displaceability:
   
   \begin{defn} \label{Epsilon Displacement Definition} 
     Fix a positive real number $\epsilon$.  A subset of a symplectic manifold $U \subset M$ is said to be $\epsilon$-shiftable if there exists a Hamiltonian diffeomorphism, $\phi$, such that 
     $$d(p, \phi(p)) \geq \epsilon \text{  } \forall p \in U.$$
         
  \end{defn}
  
  The main idea of the proof of Theorem \ref{Lipschitz Estimate} is as follows: we first reduce the theorem to the case where $G=0$.  In Theorem \ref{Epsilon Displacement Theorem} we will show that if the support of $H$ can be $\epsilon$-shifted by $\psi \in Ham(M)$, and if $d_{C^0}^{path}(Id, \phi^t_H) \leq \epsilon$ then $| \rho(H;1) | \leq \Vert \psi \Vert_{\gamma}$.  We then prove Theorem \ref{Lipschitz Estimate} by carefully constructing $\psi$.  The details of this argument are carried out in section \ref{The Lipschitz Estimate}.
   
   The following statement follows readily from Theorem \ref{Lipschitz Estimate} and the triangle inequality:
  \begin{cor} \label{Lipschitz Estimate for Gamma Norm}
  Suppose $H, G \in C^{\infty}_c([0,1] \times (M \setminus B))$.   There exist constants $C, \delta> 0$ (depending on $B$) such that if $d_{C^0}^{path}(\phi^t_G, \phi^t_H) \leq \delta$, then 
   
       $$|\gamma(G) - \gamma(H)| \leq C \, d_{C^0}^{path}(\phi^t_G, \phi^t_H).$$
  \end{cor}
  
  We will prove Theorem \ref{Lipschitz Estimate} in section \ref{The Lipschitz Estimate}.  By part \ref{Shift} of Proposition \ref{Properties of Spectral Invariants}, if $r:[0,1] \rightarrow \mathbb{R}$ is a function of time then $\gamma(H) = \gamma(H + r)$, and so $\gamma$ is not affected by normalization.  Hence, in light of Corollary \ref{Lipschitz Estimate for Gamma Norm}, it is more reasonable to expect that the map $\gamma: PHam \rightarrow \mathbb{R}$ is $C^0$-continuous.  It follows from Theorem \ref{Continuity of Gamma} that on surfaces this indeed is the case.  
  
  \subsubsection{{\bf Spectral norm v.s. $C^0$-norm}}
  In section \ref{Surfaces}, we study the relation between $\Vert \cdot \Vert_{\gamma}$ and $d_{C^0}$ on surfaces.  In \cite{H2}, Hofer compares the $C^0$-distance and the Hofer distance on $Ham_c(\mathbb{R}^{2n})$, the group of compactly supported Hamiltonian diffeomorphisms of $\mathbb{R}^{2n}$, and obtains the well known $C^0$-Energy estimate:
  $$ d_{Hofer}(\phi, \psi) \leq C d_{C^0}(\phi, \psi).$$
  
  No estimate of this kind holds on closed manifolds.  In fact, one can show that on any surface there exists a sequence of Hamiltonian diffeomorphisms which converges to the identity in $C^0$-topology, but diverges with respect to Hofer's metric. Contrary to the above fact, in section \ref{Surfaces}, we will show that:
    \begin{thm}\label{Gamma Norm}
    Let $(\Sigma, \omega)$ denote a closed surface of genus $g$ equipped with an area form.  Suppose that $\phi \in Ham(\Sigma)$.  There exist constants $C\text{, }\delta > 0 $ such that if  $d_{C^0}(Id, \phi) \leq \delta$ then 
     $$\Vert \phi \Vert_{\gamma} \leq C \, (d_{C^0}(Id, \phi))^{2^{-2g-1}}.$$ 
    
   \end{thm}
   
    \begin{rem}
     It follows readily from the above result that if $\phi, \psi \in Ham(\Sigma)$ and $d_{C^0}(\phi, \psi) \leq \delta$ then 
   $$d_{sp}(\phi, \psi) \leq C \, (d_{C^0}(\phi, \psi))^{2^{-2g-1} }.$$ 
   This establishes $C^0$-continuity of the spectral norm on surfaces.  This answers part 1 of Question 5.13 of \cite{Oh6}, in the case of surfaces.
   \end{rem}
   
   One interesting consequence of Theorem \ref{Gamma Norm} is the following result about $C^0$-continuity of the map
   $\gamma: PHam(\Sigma) \rightarrow \mathbb{R}$.  
   
   \begin{thm} \label{Continuity of Gamma}
     Suppose $\phi^t_H \in PHam(\Sigma)$, where $\Sigma$ is a surface of genus $g$.  There exist constants $C\text{, }\delta > 0 $ such that if  $d_{C^0}^{path}(Id, \phi^t_H) \leq \delta$ then 
     $$\gamma(H) \leq C \, (d_{C^0}(Id, \phi^1_H))^{2^{-2g-1}}.$$ 
   \end{thm} 
   
   The above statement no longer holds if one replaces the assumption $d_{C^0}^{path}(Id, \phi^t_H) \leq \delta$ with the weaker assumption that $d_{C^0}(Id, \phi^1_H) \leq \delta$.  To see this let $H$ denote a time independent Hamiltonian on $S^2$ whose flow rotates the sphere nearly a full turn around its central axis.  Then, $\phi^1_H$ is $C^0$-close to the identity but $\gamma(H)$ is almost $4\pi$.
   
   We will present a proof of Theorem \ref{Continuity of Gamma} in section \ref{Surfaces}.  The proof presented there follows almost immediately from the following lemma, which is probably of independent interest.
   
   \begin{lem}\label{Contractibility of small loops}
   Suppose that $\Sigma$ is a surface of genus $g$ and that $\phi^t_H \in PHam(\Sigma)$ is a loop.  There exists a constant $\delta > 0 $ such that if  $d_{C^0}^{path}(Id, \phi^t_H) \leq \delta$ then the loop $\phi^t_H$ is contractible. 
   \end{lem}
   Of course, the above lemma is only interesting in the case of $\Sigma = S^2$, as $Ham(\Sigma)$ is simply connected for other surfaces.  The proof of this lemma, which will be presented in section \ref{Surfaces}, uses Theorem \ref{Gamma Norm}.
   
  \begin{rem}\label{Connection to Viterbo}
  The statements of Theorems \ref{Lipschitz Estimate}, \ref{Gamma Norm} and \ref{Continuity of Gamma} can be translated into questions about Lagrangian submanifolds of cotangent bundles:  let $\mathcal{L}_0$ denote the space of Lagrangian submanifolds of a cotangent bundle which are Hamiltonian isotopic to the zero section.  In \cite{V1}, C. Viterbo uses generating functions to construct a distance, which we denote by $\Vert \cdot \Vert_{\gamma, \mathcal{L}_0}$, on $\mathcal{L}_0$.  In fact, Viterbo's work in \cite{V1} is a precursor to Oh and Schwarz's work on spectral invariants.  Viterbo has asked whether the distance $\Vert \cdot \Vert_{\gamma, \mathcal{L}_0}$ can be bounded by a multiple of the Hausdorff distance.  An affirmative answer to Viterbo's question would have significant consequences for his theory of symplectic homogenization \cite{V2}.  
  
  %Theorems \ref{Lipschitz Estimate}, \ref{Gamma Norm} and \ref{Continuity of Gamma} are connected to a current project of C. Viterbo: let $\mathcal{L}_0$ denote the space of Lagrangian submanifolds of a cotangent bundle which are Hamiltonian isotopic to the zero section.  In \cite{V1}, Viterbo uses generating functions to construct a distance, which we denote by $\Vert \cdot \Vert_{\gamma, \mathcal{L}_0}$, on $\mathcal{L}_0$.  In fact, Viterbo's work in \cite{V1} is a precursor to Oh and Schwarz's work on spectral invariants.  Viterbo has asked whether the distance $\Vert \cdot \Vert_{\gamma, \mathcal{L}_0}$ can be bounded by a multiple of the Hausdorff distance.  %It is proven in \cite{V2} that on the cotangent bundle of a torus the distance $\Vert \cdot \Vert_{\gamma, \mathcal{L}_0}$ is bounded by a multiple of the Hausdorff distance.  In \cite{V3}, Viterbo extends the above results to cotangent bundles of simply connected manifolds.  In fact, in the above papers, Viterbo proves a stronger theorem which has as its corollary the above mentioned results.
  
   Note that the distance $\Vert \cdot \Vert_{\gamma}$ on $Ham$ can not be bounded by $d_{C^0}$.  This is because on some compact manifolds, such as the two dimensional torus, $\Vert \cdot \Vert_{\gamma}$  is unbounded.  
  \end{rem}
  
  It would be interesting to see if results in the spirit of Theorems \ref{Gamma Norm} \& \ref{Continuity of Gamma} hold on more general symplectic manifolds.
  
  \subsubsection{{\bf Applications to the theory of Calabi quasimorphisms}}
   Assume that $M$ admits a Calabi quasimorphism $\mu :\widetilde{Ham}(M) \rightarrow \mathbb{R}$, as defined by Entov and Polterovich in \cite{EP1}.  See section \ref{Application to the theory of Calabi quasimorphisms} for the definition of $\mu$ and the necessary background.\\
    
    {\bf First Application: } Let $\eta = \mu - \widetilde{Cal}_U$.  The Calabi property of $\mu$ indicates that $\eta = 0$ if $U$ is displaceable.  If  $U$ is not displaceable, our first application of Theorem \ref{Lipschitz Estimate} establishes that the homogeneous quasimorphism {\bf $\eta$ is locally Lipschitz} with respect to $d_{C^0}^{path}$.  See Theorem \ref{Application to Calabi Quasimorphisms} for a precise statement.\\
    
    {\bf Second Application: } Denote by $B^{2n}$ the ball of radius 1 in $\mathbb{R}^{2n}$, by $Ham(B^{2n})$ the group of compactly supported Hamiltonian diffeomorphisms of $B^{2n}$, and by $\mathcal{H}(B^{2n})$ the $C^0$-closure of $Ham(B^{2n})$ inside the group of compactly supported homeomorphisms of $B^{2n}$.  In \cite{EPP} Entov, Polterovich and Py study a family of homogeneous quasimorphisms $\eta_{\delta}: Ham(B^{2n}) \rightarrow \mathbb{R}$, where $\delta$ is a parameter that ranges over $(0,1)$.  One of the main results of \cite{EPP} is that the homogeneous quasimorphisms $\eta_{\delta}$ are $C^0$-continuous and hence, by general properties of homogeneous quasimorphisms, extend continuously to $\mathcal{H}(B^{2n})$; see Theorem 1.1 and Propositions 1.3, 1.4 and 1.9 in \cite{EPP}.  
    
    In Theorem \ref{Local Lipschitz Continuity of EPP's Quasimorphisms} we improve the results in \cite{EPP} by obtaining an estimate which shows firstly, that the quasimorphisms {\bf $\eta_{\delta}$ are locally Lipschitz continuous} with respect to $d_{C^0}$ on $Ham(B^{2n})$, and secondly, that each {\bf $\eta_{\delta}$ extends to a locally Lipschitz continuous map} on $\mathcal{H}(B^{2n})$.  See Theorem \ref{Local Lipschitz Continuity of EPP's Quasimorphisms} in section \ref{Application to Calabi Quasimorphisms} for proofs and precise statements.
  
\subsubsection{{\bf Spectral Hamiltonian Homeomorphisms and a Question of Y.-G. Oh}} 
     %Symplectic homeomorphisms are the $C^0$ generalizations of symplectomorphisms.  
     The set of symplectic homeomorphisms, denoted by  $\displaystyle Sympeo(M, \omega)$, is the $C^0$-closure of $Symp(M, \omega)$ in the group of homeomorphisms of $M$.  We denote by $Sympeo_0(M, \omega)$ the path connected component of the identity in $Sympeo(M, \omega)$.  We should point out that $Sympeo(D^2) = Sympeo_0(D^2)$. 
     
   Recently, Oh used spectral invariants to introduce new $C^0$ generalizations of $Ham$ and $PHam$, denoted by $Hameo_{sp}$ and $PHameo_{sp}$, respectively.  The reader should keep in mind that it is not known whether $Hameo_{sp}$ and $PHameo_{sp}$ are closed under composition.  These objects will be defined in section \ref{Application to Spectral Hamiltonian Homeomorphisms}.  Oh asked the following question:
   
   \begin{que}(Oh's Question) \label{Oh's Question}
   Suppose that $M=S^2$ or $D^2$.  Is $Hameo_{sp}(M)$ a proper subset of $Sympeo_0(M)$?  
   \end{que} 
   
   An affirmative answer to the above question would settle the longstanding open problem regarding the simpleness v.s. non-simpleness of the groups $Sympeo_0(S^2)$ and $Sympeo_0(D^2)$.  This is because $Sympeo_0(S^2)$ and $Sympeo_0(D^2)$ have normal subgroups which are contained in 
   $\displaystyle Hameo_{sp}(S^2)$ and $Hameo_{sp}(D^2)$, respectively.

   In section \ref{Application to Spectral Hamiltonian Homeomorphisms}, we will answer the above question in the case of $D^2$ by showing that $Hameo_{sp}(D^2) = Sympeo_0(D^2)$.  In the same section we present two results regarding uniqueness issues for spectral Hamiltonian homeomorphisms.

\section{{\bf Proof of the $C^0$-spectral estimate}} \label{The Lipschitz Estimate}
 
  The main goal of this section is to prove Theorem \ref{Lipschitz Estimate}.  
  
  Recall that we say $U \subset M$ is $\epsilon$-shiftable if there exists a Hamiltonian diffeomorphism, $\phi$, such that 
     $d(p, \phi(p)) \geq \epsilon \text{  } \forall p \in U$; see Definition \ref{Epsilon Displacement Definition}.
  
  \begin{rem} An $\epsilon$-shiftable set is not necessarily displaceable.  However, every compact displaceable set is $\epsilon$-shiftable for a sufficiently small $\epsilon$.  It was pointed out to the author, by the referee of this paper, that there exist displaceable sets which can not be $\epsilon$-shifted for any $\epsilon$.  For example, let $A$ denote a countable dense subset of any closed symplectic manifold $M$.  $A$ can not be $\epsilon$-shifted because otherwise we would obtain a Hamiltonian diffeomorphism of $M$ without any fixed points.  However, $A$ is displaceable: it can be checked that if $H$ is a Morse function without any critical points in $A$, then the set $\{ t \in \mathbb{R}: \phi^t_H(A) \cap A \neq \emptyset \}$ is countable. 
  \end{rem}
  
  The support of a time dependent Hamiltonian is defined by $\displaystyle Supp(H)= \cup_{t \in [0,1]} supp(H_t)$.  The following theorem is the main reason for the introduction of Definition \ref{Epsilon Displacement Definition}.  
  \begin{thm} \label{Epsilon Displacement Theorem}
   Suppose that the support of a Hamiltonian $H$ can be $\epsilon$-shifted by $\psi \in Ham(M)$.  If $d_{C^0}^{path}(Id, \phi^t_H) < \epsilon$, then $$|\rho(H;1) | \leq \Vert \psi \Vert_{\gamma}.$$
  \end{thm}

We will now prove Theorem \ref{Lipschitz Estimate}, using the above statement.
  \begin{proof} (Theorem \ref{Lipschitz Estimate})
  
     We will prove the result in two steps.
     
    {\bf Step 1.}  Suppose $\phi^t_G =Id, a=1$.  
  
     We have to show that there exist constants $C$ and $\delta > 0$ such that if $d_{C^0}^{path}(Id, \phi^t_H) < \delta$ then 
     
     \begin{equation}\label{Step 1}
     |\rho(H;1)| \leq C \, d_{C^0}^{path}(Id, \phi^t_H).
      \end{equation}

    To establish (\ref{Step 1}), pick a Morse function $K$ with critical points contained in $B$, and denote by $X_K$ the Hamiltonian vector field of $K$.  $X_K$ is non-vanishing on the compact set $M\setminus B$.  Let $C_1 : = \inf \{\Vert X_K(x)\Vert: x \in M \setminus B\}$.   Using the compactness of $M \setminus B$, one can show that there exists a sufficiently small $r > 0$, such that for each $s\in [0, r]$ the Hamiltonian diffeomorphism $\phi^s_K\text{ } \frac{C_1 s}{2}$-shifts the set $M\setminus B$.
    
    Take $H \in C^{\infty}_c([0,1] \times (M \setminus B))$ such that $d_{C^0}^{path}(Id, \phi^t_H) < \frac{C_1 r}{2}$.  Then, by the previous paragraph, for $s \in (\frac{2}{C_1} d_{C^0}^{path}(Id, \phi^t_H), r]$ the Hamiltonian diffeomorphism $\phi^s_K\text{ } \frac{C_1s}{2}$-shifts the support of $H$.  Also, note that $d_{C^0}^{path}(Id, \phi^t_H) < \frac{C_1s}{2}$.  Therefore, Theorem \ref{Epsilon Displacement Theorem} implies that
    $$|\rho(H;1)| < \Vert \phi^s_K \Vert_{\gamma}.$$
    Because $\Vert \phi^s_K \Vert_{\gamma} \leq  \Vert s K \Vert_{(1,\infty)} = s \Vert K \Vert_{(1,\infty)}$ we have  $$|\rho(H;1)| <  s \Vert K \Vert_{(1,\infty)}.$$
   
    Since the above inequality holds for $s \in (\frac{2}{C_1} d_{C^0}^{path}(Id, \phi^t_H), r]$, we get that
    
                                 $$|\rho(H;1) |\leq  \frac{2}{C_1} d_{C^0}^{path}(Id, \phi^t_H) \Vert K \Vert_{(1,\infty)}.$$
    The estimate (\ref{Step 1}) follows, with $C:= \frac{2}{C_1}  \Vert K \Vert_{(1, \infty)}$ and $\delta := \frac{C_1 r}{2}$.
    
    {\bf Step 2.} No assumptions on $\phi^t_G$ or $a$.
    
    We use the constants $\delta$ and $C$ from the first step.  Recall that by definition $d_{C^0}^{path}(\phi^t_G, \phi^t_H) = d_{C^0}^{path}(Id, \phi^{-t}_G \phi^t_H)$.  If  $d_{C^0}^{path}(\phi^t_G, \phi^t_H) < \delta$, then it follows from the first step that  $$\rho(\bar{G}\#H;1) \leq C d_{C^0}^{path}(\phi^t_G, \phi^t_H).$$
  
  By the triangle inequality for spectral invariants we have $\rho(H;a) - \rho(G;a) \leq \rho(\bar{G}\#H;1)$.  Hence, we get
  $$ \rho(H;a) - \rho(G;a) \leq C d_{C^0}^{path}(\phi^t_G, \phi^t_H).$$
  Similarly, we get $ \rho(G;a) - \rho(H;a) \leq C d_{C^0}^{path}(\phi^t_H, \phi^t_G)$, which combined with the previous inequality implies the result.
  \end{proof}
  
  We will now provide a proof for Theorem \ref{Epsilon Displacement Theorem}.    
  \begin{proof}(Theorem \ref{Epsilon Displacement Theorem})
   Observe that it is sufficient to show that the assumptions of the theorem imply that 
   $$ \rho(H;1) \leq \Vert \psi \Vert_{\gamma}.$$
   Indeed, if the above statement holds then we get $\rho(\bar{H}; 1)  \leq \Vert \psi \Vert_{\gamma}.$  This is because $d_{C^0}^{path}(Id, \phi^t_H) = d_{C^0}^{path}(Id, \phi^{t}_{\bar{H}})$.  By the triangle inequality we have $-\rho(H;1) \leq \rho(\bar{H}; 1) $, and thus $-\rho(H;1) \leq \Vert \psi \Vert_{\gamma}$.  
  
   We may assume, by slightly $C^{\infty}$-perturbing $\psi$, if needed, that it is non-degenerate.
   Let $K$ denote a generating Hamiltonian for $\psi$, i.e., $\psi = \phi^1_K$.  
   
   We claim that $Fix(\phi^t_H \phi^1_K) = Fix(\phi^1_K)$ for each $t \in [0,1]$.  Indeed, if $p \in Fix(\phi^1_K)$, then $p$ can not belong to the support of $H$ because $\phi^1_K$ moves every point in the support of $H$ by a distance of at least $\epsilon$.  Hence, for $p \in Fix(\phi^1_K)$ we have $\phi^t_H \phi^1_K(p) = \phi^t_H(p)=p.$  This shows that, $ Fix(\phi^1_K) \subset Fix(\phi^t_H\phi^1_K)$.  
   
   To show the other containment take a point $p$ in $Fix(\phi^t_H\phi^1_K)$.  First, for a contradiction, suppose that $p \in Supp(H)$. Then $d(p, \phi^1_K(p)) > \epsilon > d_{C^0}^{path}(Id, \phi^t_H)$, so $\phi^t_H$ can not move $\phi^1_K(p)$ back to $p$ and hence, $p$ can not be a fixed point of $\phi^t_H\phi^1_K$.  Next, we will show that $\phi^1_K(p) \notin Supp(H)$.  If $\phi^1_K(p) \in Supp(H)$ then $\phi^t_H\phi^1_K(p) \in Supp(H)$, which in turn implies that $\phi^t_H\phi^1_K(p) \neq p$ because $p \notin Supp(H)$.  Since $\phi^1_K(p) \notin Supp(H)$ we get that $\phi^t_H\phi^1_K(p) = \phi^1_K(p)$.  Thus, $p \in Fix(\phi^1_K)$.
   
   Note that the above argument implies that $Fix(\phi^1_K) \cap Supp(H) = \emptyset$.   
   The result then follows from Proposition \ref{Usher Generalization}, stated and proven below.    
  \end{proof}  
  
   The following proposition is a variation of previously obtained results by Entov-Polterovich\cite{EP1}, Ostrover\cite{O}, and Usher\cite{U1}.  Our proof follows the argument in \cite{U1}.

   \begin{prop}\label{Usher Generalization} (See \cite{EP1} Proposition 3.3, \cite{O} Proposition 2.6, \cite{U1} Proposition 3.1)
   Suppose that $H, \text{ }K: [0,1]\times M \rightarrow \mathbb{R}$ are two Hamiltonians with the following properties:
   \begin{enumerate}
   \item $\phi^1_K$ is non-degenerate,
   \item $Fix(\phi^t_H\circ\phi^1_K) = Fix(\phi^1_K)$ for each $t \in [0,1]$, and
   \item $Fix(\phi^1_K) \cap Supp(H) = \emptyset$.
   \end{enumerate}
   Then, $$\rho(H;1) \leq \gamma(K).$$ 
   \end{prop}
   \begin{proof}
   Note that the 2nd and the 3rd assumptions imply that for each $t \in [0,1]$ the Hamiltonian diffeomorphism $\phi^t_H\circ\phi^1_K$ coincides with $\phi^1_K$ on a neighborhood of all its fixed points.  Hence, it follows, from the non-degeneracy of $\phi^1_K$, that $\phi^t_H\circ\phi^1_K$ is non-degenerate as well.  We may also assume that $K$ is normalized because, as was mentioned earlier, $\gamma(\cdot)$ does not distinguish between Hamiltonians that differ by a function of time.
   
   Let $\tilde{H}$ denote the normalization of $H$.  Clearly, $\tilde{H}(t,x) = H(t,x) - c(t)$, where $c(t)= \frac{\int_{M} H(t,x) \omega^n}{vol(M)}$.  
   
   Let $\alpha: [0, \frac{1}{2}] \rightarrow [0,1]$ denote a smooth, non-decreasing map from $[0, \frac{1}{2}]$ onto $[0,1]$ which equals zero on a neighborhood of zero, and equals $1$ on a neighborhood of $\frac{1}{2}$. Let $$L_s(t,x) =  \left\{ \begin{array}{ll}
         \alpha^{'}(t)K(\alpha(t),x) & \mbox{if $0 \leq t \leq \frac{1}{2}$};\\
         s\alpha^{'}(t-\frac{1}{2})\tilde{H}(s\alpha(t-\frac{1}{2}),x) & \mbox{if $\frac{1}{2} \leq t \leq 1 $}.\end{array} \right. $$
   Then, $\phi^1_{L_s} = \phi^s_{\tilde{H}} \circ \phi^1_{K}$.  This Hamiltonian diffeomorphism is non-degenerate by the discussion in the first paragraph of the proof.  Hence, $\rho(L_s;1) \in Spec(L_s)$, by the spectrality property of spectral invariants.  
   
   Next, we'll show that $\rho(L_s;1) = \rho(K;1) + \int_{0}^{s}c(t) dt$.  Let $[z,u] \in Crit(\mathcal{A}_{L_s})$.  Then, the second and the third assumptions imply that $$z(t) = \left\{ \begin{array} {ll} 
                                         \phi_K^{\alpha(t)}(z_0) & \mbox{if $0 \leq t \leq \frac{1}{2}$};\\
                                         z_0& \mbox{if $\frac{1}{2} \leq t \leq 1 $}.\end{array} \right. $$
                                         
   where $z_0 = z(0)$.  Thus, $$\mathcal{A}_{L_s}([z,u]) = - \int L_s(t,z(t))dt \text{ }- \int_{u} \omega$$
         $$ = - \int_{0}^{\frac{1}{2}} \alpha^{'}(t)K(\alpha(t), \phi_K^{\alpha(t)}(z_0)) dt \text{ }
             - \int_{\frac{1}{2}}^{1} s\alpha^{'}(t-\frac{1}{2})\tilde{H}(s\alpha(t-\frac{1}{2}),z_0) dt \text{ }
              - \int_{u} \omega $$                
          $$ = - \int_{0}^{1} K(t, \phi_K^t(z_0)) dt \text{ }
             - \int_{\frac{1}{2}}^{1} s\alpha^{'}(t - \frac{1}{2}) c(s\alpha(t - \frac{1}{2}),z_0) dt \text{ }- \int_{u} \omega $$
          $$ = - \int_{0}^{1} K(t, \phi_K^t(z_0) ) dt \text{ } - \int_{u} \omega \text{ } + \int_{0}^{s}  c(t) dt$$
          $$ = \mathcal{A}_{K}([\phi_K^t(z_0),u])\text{ } + \int_{0}^{s}  c(t) dt.$$
          
    So, we conclude that $Spec(L_s) = Spec(K) + \int_{0}^{s}c(t) dt$.  The continuous function $\rho(L_s;1) - \int_{0}^{s}c(t) dt$ (as a functions of s) takes values in the nowhere dense, measure zero set $Spec(K)$, and therefore it is constant, i.e., $\rho(L_s;1) = \rho(L_0;1) + \int_{0}^{s}c(t) dt = \rho(K;1) + \int_{0}^{s}c(t) dt$.  Here, we have used the homotopy invariance property of spectral invariants to conclude that $\rho(L_0;1) = \rho(K;1).$
    
    The Hamiltonian paths $\phi^t_{L_1}$ and $\phi^t_{\tilde{H}\#K}$ are homotopic rel. endpoints, and the Hamiltonians $L_1, \text{ }\tilde{H}\#K$ are both normalized.  Thus, by the homotopy invariance property, 
    $$\rho(\tilde{H}\#K;1) = \rho(L_1;1) = \rho(K;1) + \int_{0}^{1}c(t) dt.$$
    
    We then have:  $\rho(\tilde{H};1) \leq \rho(\tilde{H}\#K;1) + \rho(\bar{K};1) = \rho(K;1) + \rho(\bar{K};1)+ \int_{0}^{1}c(t) dt.$  Because $\rho(\tilde{H};1) = \rho(H;1) + \int_{0}^{1}c(t) dt$, by part \ref{Shift} of Proposition \ref{Properties of Spectral Invariants}, we conclude that
    $$ \rho(H;1) \leq \rho(K;1) + \rho(\bar{K};1) = \gamma(K).$$
   \end{proof}
   
   \begin{exmp} \label{Discontinuity Example}
   In Theorem \ref{Lipschitz Estimate}, we consider Hamiltonian paths $\phi^t_H$ which fix the points of a ball $B$ for all time, i.e., $\phi^t_H(p)=p \;\; \forall (t,p) \in [0,1]\times B$.  As one can see from the statement of the theorem, the generating Hamiltonian $H$ is taken to be the unique Hamiltonian which vanishes on $B$, i.e., $H(t, p) = 0 \;\; \forall (t,p) \in [0,1]\times B$.  This is our way of ``normalizing" generating Hamiltonians for such Hamiltonian paths.  The usual normalization procedure is different than ours; it requires the generating Hamiltonian to satisfy:  $\int_{M}H(t,\cdot)\omega^n = 0$ for each $t\in [0,1]$.  In this example, we demonstrate that Theorem \ref{Lipschitz Estimate} does not hold if the generating Hamiltonians are required to be normalized in the usual sense of normalization.
   
 Let $F$ be a smooth, time-independent Hamiltonian, supported inside a Darboux chart $(U,x,y)$ such that $\int_{M}F=1.$ Let $F_k = k^{2n}F(kx, ky),$ where $2n$ is the dimension of the manifold.  Note that $Supp(F_k)$ shrinks to a point and thus $d_{C^0}^{path}(Id, \phi^t_{F_k})$ converges to $0$.   Also, note that $\int_{M} F_k = 1$, so these Hamiltonians are not normalized.  
   
   Corollary 3.3 in \cite{U1} states that $\rho(F_k;1) \leq e(Supp(F_k))$, where $e(Supp(F_k))$ is the displacement energy of support of $F_k$.  Applying the above mentioned result to $\bar{F_k}$ we get $\rho(\bar{F_k};1) \leq e(Supp(F_k))$, which combined with the fact that $0 \leq \rho(F_k;1) + \rho(\bar{F_k};1)$, implies $- \rho(F_k;1)  \leq e(Supp(F_k))$.  Hence, we have $$| \rho(F_k;1) |  \leq e(Supp(F_k)).$$
   
   Similarly, one can show that 
   $\left| \rho(\bar{F_k};1) \right| \leq e(supp(F_k)) $. 
   
   Since the sets $Supp(F_k)$ shrink to a point, $e(Supp(F_k))$ converges to $0$.   Thus, $\left| \rho(F_k;1) \right| $ and $\left| \rho(\bar{F_k};1) \right|$ both converge to zero.    
   
   Let $H_k$ be the Hamiltonian obtained by normalizing $F_k$, i.e., $H_k = F_k - 1$.  Then, 
     $$\lim_{k\to \infty} \rho(H_k; 1)= \lim_{k\to \infty}\rho(F_k - 1;1)  =  \lim_{k\to \infty}\rho(F_k;1) + 1= 1.$$
     
     Similarly, we see that $$\lim_{k\to \infty} \rho(\bar{H}_k;1) = -1.$$
\end{exmp}

   \begin{exmp} \label{Sharpness of the Lipschitz Estimate}
      In this example, we will show that the (locally) Lipschitz estimate obtained in Theorem \ref{Lipschitz Estimate} is sharp in the sense that it can not be improved to a (locally) H\"older estimate of H\"older exponent larger than $1$, i.e., the following estimate, for $H$ as in Theorem \ref{Lipschitz Estimate},  can only hold if $\alpha \leq 1$:
      
     \begin{equation}\label{eq: Holder Estimate}
 	      |\rho(H;1)| \leq C \, (d_{C^0}^{path}(Id, \phi^t_H))^{\alpha}.
 	   \end{equation}
 	   
      Our example is for the case of surfaces, but it can easily be generalized to higher dimensions.  Let $U$ denote a Darboux chart on a surface $(\Sigma, \omega)$, and assume that  $\omega = r dr \wedge d\theta$, in $U$.  Let $a$ be a small enough positive number such that (an embedding of) the disk of radius $a$ is contained in $U$.  Pick a smooth function $h:[0,a] \rightarrow \mathbb{R}$ such that $h \equiv -a$ on $[0, \epsilon]$, $h$ is increasing on $(\epsilon, a-\epsilon)$ , and $h \equiv 0$ on ($ a -\epsilon, a)$, where $\epsilon$ is picked to be sufficiently small.  Extend $h$ to $\Sigma$ by setting it to be zero outside the disk of radius $a$.  Note that $X_h(r, \theta) := h'(r) \frac{\partial}{\partial \theta}$, and hence $\Vert X_h(r, \theta) \Vert = r |h'(r)| \leq C$, for some constant $C$.  This implies that for each $s\in [0,1]$ we have $d_{C^0}^{path}(Id,  \phi^s_h) \leq s\,C$.  Define a sequence of Hamiltonians $H_i : = \frac{1}{i} h$.  The above discussion implies that
       $$d_{C^0}^{path}(Id, \phi^t_{H_i}) \leq \frac{1}{i} \, C.$$
      
        Observe that for $i$ large enough $H_i$ is $C^{\infty}$-small, and hence it has no non-trivial contractible periodic orbits of period at most $1$.  Proposition 4.1 in \cite{U1}, states that if a Hamiltonian, $K$, has no non-trivial contractible periodic orbits of period at most one then $\rho(K;1) = - \min_{M} K$.  Hence, $$ \rho(H_i;1) = - \min_{M} H_i = \frac{a}{i}.$$
      
      We conclude that the estimate (\ref{eq: Holder Estimate}) can only hold for $\alpha \leq 1$.

   \end{exmp}
   
\section{{\bf The case of surfaces: spectral norm v.s. $C^0$-norm}} \label{Surfaces}

Our main objective in this section is to prove Theorems \ref{Gamma Norm}, \ref{Continuity of Gamma} and Lemma \ref{Contractibility of small loops}.  Throughout this section $(\Sigma, \omega)$ denotes a surface equipped with an area form $\omega$.  A disk in $\Sigma$ is the image of an area preserving embedding of $D^2_r:= \{(x,y) \in \mathbb{R}^2 : x^2 + y^2 \leq r \}.$

\subsection{{\bf A fragmentation theorem}}

 To prove Theorem \ref{Gamma Norm} we will employ a fragmentation theorem for $C^0$-small Hamiltonian diffeomorphisms of a surface.  In the case of a surface with boundary, $Ham(\Sigma)$ denotes the group of Hamiltonian diffeomorphisms generated by Hamiltonians that vanish near $\partial \Sigma$.  Also, recall that if a surface $\Sigma$ has non-empty boundary then it can be obtained by attaching a number of $1$-handles to a disk. 
  
  \begin{prop} ($C^0$-Fragmentation, see \cite{EPP} section 1.6.2) \label{Fragmentation}
   Let $\Sigma$ denote a compact surface.  There exists a $C^0$-neighborhood $\nu$ of the identity in $Ham(\Sigma)$ and a finite covering of $\Sigma$ consisting of $N$ disks $(D_i)_{1 \leq i \leq N}$ with the property that any $\phi \in \nu$ can be written as a composition $\phi = \phi_1 \cdots \phi_N$, where each $\phi_i$ is supported in one of the disks $D_j$ and satisfies the estimate 
                           $$d_{C^0}(Id, \phi_i) \leq C \, (d_{C^0}(Id, \phi))^{2^{1-N}}, $$
   where $C$ is a constant.  Furthermore, 
   \begin{enumerate}
   \item  If $\partial \Sigma \neq \emptyset,$ then $N= l+1$, where $l$ denotes the number of $1$-handles needed to obtain $\Sigma$ from a disk.
   \item If $\partial \Sigma = \emptyset,$ then $N = 2g+2$.  
   \end{enumerate}

  \end{prop}
  
   This result is a slight modification of a statement that appears in section 1.6 of \cite{EPP}.  We will discuss the proof of this result in section \ref{Proof of the Fragmentation Theorem }.  
   
  \subsection{{\bf Proof of Theorem \ref{Gamma Norm}}}
  
  The proof of Theorem \ref{Gamma Norm} will use the following lemma:
  \begin{lem}\label{Local Path Connectedness}
      
      Suppose $\psi \in Ham(B^{2n}_r)$.  There exists a Hamiltonian $H: [0,1] \times B^{2n}_r \rightarrow \mathbb{R}$ such that $\psi = \phi^1_H$ and $d_{C^0}^{path}(Id, \phi^t_H) \leq d_{C^0}(Id, \psi)$.
  \end{lem}
  
  \begin{rem}\label{Local Path Connectedness General}
      Let $B_r$ denote the image of a symplectic embedding of $B^{2n}_r$ into $M$. Suppose $\psi \in Ham(B_r)$, i.e., there exists a Hamiltonian $G$ whose support is contained in $B_r$ and $\psi = \phi^1_G$.
      One can easily check that the proof of Lemma \ref{Local Path Connectedness} can be adapted to obtain the following statement: 
      
      There exists a Hamiltonian $H$ supported in $B_r$ such that $\psi = \phi^1_H$ and $d_{C^0}^{path}(Id, \phi^t_H) \leq C \, d_{C^0}(Id, \psi),$ where $d_{C^0}$ denotes a $C^0$-distance on $M$ and $C$ is an appropriately chosen constant.
    \end{rem}
  
  Postponing the proof of the above lemma, we now prove Theorem \ref{Gamma Norm}:

  \begin{proof} (Theorem \ref{Gamma Norm})
     We pick $\delta$ small enough so that we have $\phi \in \nu$, where $\nu$ is the neighborhood of the identity from Proposition \ref{Fragmentation}.  Applying Proposition \ref{Fragmentation}, we obtain disks $(D_i)_{1 \leq i \leq 2g+2}$ and Hamiltonian diffeomorphisms $\phi_i \in Ham(D_j)$ such that $\phi = \phi_1 \cdots \phi_{2g+2}$, and 
     \begin{equation} \label{Estimate From Fragmentation}
        d_{C^0}(Id, \phi_i) \leq C_1 \, (d_{C^0}(Id, \phi))^{2^{-2g-1}}.
    \end{equation}
     
      Lemma \ref{Local Path Connectedness} and Remark \ref{Local Path Connectedness General} imply that we can find Hamiltonians $H_1, \cdots, H_{2g+2}$ such that $H_i$ is supported in the same disk as $\phi_i$, $\phi_i = \phi^1_{H_i}$, and 
     
     \begin{equation} \label{Estimate From Path Connectedness}
       d_{C^0}^{path}(Id, \phi^t_{H_i}) \leq C_2 \, d_{C^0}(Id, \phi_i).
     \end{equation}
  
     Assuming $\delta$ is sufficiently small, we can apply Corollary \ref{Lipschitz Estimate for Gamma Norm} to each Hamiltonian $H_i$ and obtain
      $$\gamma(H_i) \leq A_i \, d_{C^0}^{path}(Id, \phi^t_{H_i}),$$
     where $A_i$ is a constant depending on the disk $D_j$ which contains the support of $H_i$.  Combine the above inequality with the estimates (\ref{Estimate From Fragmentation}) and (\ref{Estimate From Path Connectedness}) to get
  \begin{equation} \label{Combined Estimate for the Hamiltonians}
   \gamma(H_i) \leq C \, (d_{C^0}(Id, \phi))^{2^{-2g-1}},
  \end{equation}
   where $C$ is an appropriately chosen constant.  Let $H:= H_1 \# \cdots \# H_{2g+2}$, so that $\phi = \phi^1_H$.  The triangle inequality for spectral invariants and estimate (\ref{Combined Estimate for the Hamiltonians}) imply that $$ \gamma(H) \leq \Sigma_{i=1}^{2g+2} \gamma(H_i) \leq (2g+2) C\, (d_{C^0}(Id, \phi))^{2^{-2g-1}},$$
   from which the result follows.
  \end{proof}
  
  \begin{rem} \label{Local Path Connectedness on Surfaces}
    Observe that in the above proof we have shown that if $\phi \in \nu$, where $\nu$ is the neighborhood from Proposition \ref{Fragmentation}, then there exists a Hamiltonian $H$ such that $$  \phi^1_H = \phi  \text{, }  d_{C^0}^{path}(Id, \phi^t_H) \leq A \, (d_{C^0}(Id, \phi))^{2^{-2g-1}} \text{, and } \gamma(H) \leq C \, d_{C^0}(Id, \phi))^{2^{-2g-1}},$$
    for appropriately chosen constants $A, \; C$.
  \end{rem}
  
  Finally, we give a proof for Lemma \ref{Local Path Connectedness}.
  \begin{proof}(Lemma \ref{Local Path Connectedness})
  WLOG, we may assume that $r=1$.  Indeed, the proof presented below for $B^{2n}_1$ can be rescaled to provide a proof for any value of $r$.  Take a positive constant $\epsilon$ such that $\epsilon \leq \frac{d_{C^0}(Id, \psi)}{2}$.
    For each $s \in [\epsilon, 1]$ we define a Hamiltonian diffeomorphism $\phi^s$ as follows:
    $$\phi^s(x) =  \left\{ \begin{array}{ll}
         s\psi(\frac{x}{s}) & \mbox{if $|x| \leq s $};\\
         x & \mbox{if $|x| \geq s$}.\end{array} \right. $$
    A simple computation shows that $\phi^s \in Ham(B^{2n}_r)$, and in fact if $G(t,x)$ is a Hamiltonian, supported in $B^{2n}_1$, which generates $\psi$ then $\phi^s$ is the time $1$ map of the flow of the following Hamiltonian:
    $$G_s(t,x) =  \left\{ \begin{array}{ll}
         s^2G(t,\frac{x}{s})& \mbox{if $|x| \leq s $};\\
         0 & \mbox{if $|x| \geq s$}.\end{array} \right. $$
         
    Note that $\phi^s$ is obtained by rescaling $\psi$ by a factor of $s$ and hence it can easily be checked that for each $s \in [\epsilon, 1]$ we have 
    \begin{equation}\label{The Path is Small}
      d_{C^0}(Id, \phi^s) \leq s \, d_{C^0}(Id, \psi).
    \end{equation}
    
    It remains to define $\phi^s$ for $s\in[0,\epsilon]$.  We do so by the formula:$$\phi^s(x) := \phi^{\frac{s}{\epsilon}}_{G_{\epsilon}}.$$
    We obtain $\phi^s$ for all $s \in [0,1]$ by smoothly concatenating the two paths $\phi^s|_{s\in[0,\epsilon]}$ and $\phi^s|_{s\in[\epsilon,1]}$.  Note that the Hamiltonian path $\phi^s \text{ } (s\in [0,\epsilon])$ is supported in the ball of radius $\epsilon$ and hence its distance from the identity is bounded by $2 \epsilon$ which is smaller than $d_{C^0}(Id, \psi)$.  This combined with (\ref{The Path is Small}) implies that the whole path $\phi^s \text{ } (s\in [0,1])$ satisfies the inequality $$d_{C^0}(Id, \phi^s) \leq d_{C^0}(Id, \psi).$$  %The path $\phi^s$ is not necessarily smooth at $s = \epsilon$, however we can reparametrize it so that it becomes smooth.  
    Let $H$ be the Hamiltonian that generates $\phi^s$.  Clearly $H$ satisfies all the required conditions.
  \end{proof}
  
  \subsection{{\bf Proofs for Theorem \ref{Continuity of Gamma} and Lemma \ref{Contractibility of small loops}}}
  
  We will now provide proofs for Theorem \ref{Continuity of Gamma} and Lemma \ref{Contractibility of small loops}.
  
   \begin{proof}(Proof of Theorem \ref{Continuity of Gamma})
   It is well known that $Ham(\Sigma)$ is simply connected if $\Sigma$ has positive genus.  See chapter 7 of \cite{P} for a proof of this fact.  This implies that $\Vert \phi^1_H \Vert_{\gamma} = \gamma(H)$.  Hence, in the case of surfaces of positive genus we get, from Theorem \ref{Gamma Norm}, that if $d_{C^0}(Id, \phi^1_H) \leq \delta$ then $\gamma(H)$ satisfies the required inequality.  
   
     For the rest of this proof we assume that $\Sigma = S^2$.  We pick $\delta$ such that the entire path $\phi^t_H$ lies in the neighborhood $\nu$ from Proposition \ref{Fragmentation}.  According to Remark \ref{Local Path Connectedness on Surfaces} there exists a Hamiltonian $K$ such that 
     $$\phi^1_K = \phi^1_H \text{, } d_{C^0}^{path}(Id, \phi^t_K) \leq A\, \delta^{2^{-2g-1}} \text{, and } \gamma(K) \leq  C\, d_{C^0}(Id, \phi^1_H)^{2^{-2g-1}},$$
     for some constants $A$ and $C$. We're done if we show that $\gamma(H) = \gamma(K).$  By the homotopy invariance property of spectral invariants it is sufficient to show that the following Hamiltonian loop is contractible:
     
      $$\lambda(t) =  \left\{ \begin{array}{ll}
         \phi^{2t}_H & \mbox{if $0 \leq t \leq \frac{1}{2} $};\\
         \phi_K^{1-2t} & \mbox{if $ \frac{1}{2} \leq t \leq 1$}.\end{array} \right. $$
      Note that 
              $$d_{C^0}^{path}(Id,\lambda) \leq \max(\delta, A\, \delta^{2^{-2g-1}}),$$
      hence, by picking a sufficiently small $\delta$, we can ensure that the Hamiltonian path $\lambda(t)$ is $C^0$-small enough for the application of Lemma \ref{Contractibility of small loops}, which implies that $\lambda(t)$ is indeed contractible.   
   
   \end{proof}
   
   Finally, we prove Lemma \ref{Contractibility of small loops}:
   \begin{proof}(Proof of Lemma \ref{Contractibility of small loops})
   Once again, because $Ham(\Sigma)$ is simply connected for surfaces of positive genus, we assume that $\Sigma = S^2$. 
   We set $S^2 = \{(x,y,z) \in \mathbb{R}^3 : x^2 + y^2 + z^2 = 1 \}$ and equip it with the standard area form. 
   It is well known that $\pi_1(Ham(S^2)) = \frac{\mathbb{Z}}{2\mathbb{Z}}$, with the non trivial element being the full rotation around the vertical axis; see Chapter 7 of \cite{P}.  Let $f: S^2 \rightarrow \mathbb{R}$ denote the time independent and normalized Hamiltonian generating the mentioned rotation.  One can easily check that $\rho(f;1) = \rho(\bar{f};1) = 2 \pi$, and thus $\gamma(f) = 4 \pi$.
   
   We pick $\delta$ small enough such that if $d_{C^0}(Id, \phi) \leq \delta$ then, by Theorem \ref{Gamma Norm}, $\Vert \phi \Vert_{\gamma} \leq C \, d_{C^0}(Id, \phi)^{\frac{1}{2}}.$  Now suppose that $d_{C^0}^{path}(Id, \phi^t_H) \leq \delta$.  
   We break the interval $[0,1]$ into $N$ equal parts and consider the Hamiltonian paths:
   
    $$\phi^t_{K_i} = (\phi_H^{\frac{i-1}{N}})^{-1} \phi_H^{\frac{t+i-1}{N}} \text{, } 0 \leq t \leq 1 \text{, } 1\leq i \leq N.$$ 
   Note that $\phi^t_H$ is the concatenation of the paths $\phi_H^{\frac{i-1}{N}} \phi^t_{K_i}$.  We pick $N$ large enough to ensure that the paths $\phi^t_{K_i}$ are all $C^{\infty}$ small enough to guarantee that
   $$\gamma(K_i) \leq 1.$$
   
   By Theorem \ref{Gamma Norm}, we can pick Hamiltonians $F_i \text{ } 1 \leq i \leq N$ such that 
   $$\phi^1_{F_i} = \phi_H^{\frac{i}{N}} \text{ and } \gamma(F_i) \leq 2 C\,\delta^{\frac{1}{2}}.$$
   
   Consider the Hamiltonian loops:
   
    $$\lambda_i(t) =  \left\{ \begin{array}{ll}
         \phi^{3t}_{F_{i-1}} & \mbox{if $0 \leq t \leq \frac{1}{3} $};
         \vspace{1 mm}\\
         \phi_H^{\frac{i-1}{N}} \phi_{K_i}^{3(t-\frac{1}{3})} & \mbox{if $ \frac{1}{3} \leq t \leq \frac{2}{3}$};
         \vspace{1 mm}\\
         \phi^{1- 3(t-\frac{2}{3}) }_{F_i} & \mbox{if $ \frac{2}{3} \leq t \leq 1 $},\end{array} \right. $$
         
  \noindent  where $1 \leq i \leq N$ and we assume that $F_0 = 0$.  Observe that $\phi^t_H$ is homotopic to the concatenation of the loops $\lambda_i \text{ } 1 \leq i \leq N$.  Hence, it is sufficient to show that each loop $\lambda_i$ is contractible.
    
    The loop $\lambda_i$ is homotopic to the composition
    
    $$\phi^t_{F_{i-1}}\phi^t_{K_i} (\phi^1_{F_i})^{-1} \phi^{1-t}_{F_i}.$$
    
   Furthermore, the path $(\phi^1_{F_i})^{-1} \phi^{1-t}_{F_i}$ is homotopic with fixed end point to the path $\phi^t_{\bar{F_i}}$.  Therefore, $\lambda_i$ is homotopic to
   
   $$\phi^t_{F_{i-1}} \phi^t_{K_i}  \phi^{t}_{\bar{F_i}}.$$
    
    By the triangle inequality we get that
    $$\gamma(\lambda_i) \leq \gamma(F_{i-1}) + \gamma(K_i) + \gamma(\bar{F_i})$$
    $$\leq 2 C\, \delta^{\frac{1}{2}} + 1 + 2 C\, \delta^{\frac{1}{2}} < 4 \pi,$$
    where the last inequality holds for sufficiently small values of $\delta$.  This implies that, for sufficiently small values of $\delta$, $\lambda_i$ is not homotopic to the full rotation around the central axis of $S^2$, and hence it must be contractible.
   
   \end{proof}
  
\subsection{Proof of the Fragmentation Theorem} \label{Proof of the Fragmentation Theorem }

   This section contains a sketch of the proof of Proposition \ref{Fragmentation}.  This fragmentation result is a slight modification of an assertion that appears in section 1.6.2 of \cite{EPP}.  The mentioned assertion is labeled by $(*)$ in \cite{EPP}. Proposition \ref{Fragmentation} can be extracted from the (very technical) proof that is presented there by making a few modifications.  Hence, we will only outline the argument presented in \cite{EPP} and mention the changes that must be made to that argument to obtain Proposition \ref{Fragmentation}.  In order to make it easier for readers to compare the proof presented here and the original proof of \cite{EPP}
 , we will try to follow the notation and format of the argument in \cite{EPP} as closely as possible.

\begin{center}
{\bf Moser's Trick:}
\end{center}
    The argument in \cite{EPP} repeatedly uses a variation of Moser's trick; see Proposition 5 in section 1.6.1 of \cite{EPP}.  Here we modify part (iii) of that proposition as follows:
   
  {\it Let $\Sigma$ be a compact connected oriented surface, possibly with a non-empty boundary $\partial \Sigma$, and let $\omega_1$,
$\omega_2$ be two area-forms on $\Sigma$. Assume that $\int_\Sigma \omega_1 = \int_\Sigma \omega_2 $. If $\partial\Sigma \neq
\emptyset$, we also assume that the forms $\omega_1$ and $\omega_2$ coincide on $\partial\Sigma$.

Then there exists a diffeomorphism $f: \Sigma\to \Sigma$, isotopic to the identity, such that $f^* \omega_2 = \omega_1$. Moreover,
$f$ can be chosen to satisfy the following properties:

\medskip \noindent (i) If $\partial\Sigma \neq \emptyset$, then $f$ is the identity on $\partial\Sigma$, and if $\omega_1$ and
$\omega_2$ coincide near $\partial\Sigma$, then $f$ is the identity near $\partial\Sigma$.

\medskip
\noindent (ii) If $\Sigma$ is partitioned into polygons (with piecewise smooth boundaries), so that
$\omega_2-\omega_1$ is zero on the 1-skeleton $\Gamma$ of the partition and the integrals of $\omega_1$ and $\omega_2$ over each
polygon are equal, then $f$ can be chosen to be the identity on $\Gamma$.

\medskip 
\noindent (iii) Suppose that $\omega_2 = \chi \omega_1$ for a function $\chi$.  The diffeomorphism $f$ can be chosen
to satisfy the following estimate: $$d_{C^0}(Id, f) \leq C \Vert \chi -1 \Vert_{C^0},$$ for some $C>0$.  
Here, $\Vert \cdot \Vert_{C^0}$ denotes the standard sup norm on functions.

}

\noindent
{\bf Proof:}  Following Moser's trick we consider the path of symplectic forms $\omega_t = \omega_1 + t(\omega_2- \omega_1)$.  
 The form $\omega_2 - \omega_1$ is exact.  Pick a 1-form $\sigma$ such that $d\sigma = \omega_2 - \omega_1$.  
Let $f$ be the time-1 map of the time dependent vector field $X_t$  defined by: $\iota_{X_t} \omega_t =  \sigma$.
Parts (i) and (ii) of the above statement are identical to what appears in \cite{EPP} and are proven there.  To prove 
Part (iii) we must ensure that the 1-form $\sigma$ satisfies $\Vert \sigma \Vert \leq C \Vert \chi -1 \Vert_{C^0}$.
 Lemma 1 of \cite{M}, reduces this to the case where $\omega_2 - \omega_1$ is supported in a rectangle.  In that case
 one can construct $\sigma$ and show that it satisfies the required estimate.

\begin{center}
{\bf The Extension Lemmas} 
\end{center}
 
We will need the following two extension lemmas to prove our fragmentation result.  These lemmas
are modifications of Lemmas 2 \& 3 from section 1.6.1 of \cite{EPP}.

\medskip

\noindent 
{\bf Area-preserving extension lemma for disks:} 
Let $D_1\subset D_2 \subset D\subset \mathbb{R}^2$ be closed disks such that $D_1 \subset {\rm Interior}\ (D_2)\subset D_2\subset {\rm Interior}\ (D)$. Let $\phi: D_2\to D$ be a smooth area-preserving embedding (we assume $D$ is equipped with some area form). If  $\phi$ is sufficiently $C^0$-small, then there exists $\psi\in Ham (D)$ such that 
$$ \left. \psi\right|_{D_1} = \phi \;\;\;\; {\it and } \;\;\;\;  d_{C^0}(Id, \psi) \leq (d_{C^0}(Id, \phi))^{\frac{1}{2}} .$$

\medskip

\noindent 
{\bf Area-preserving extension lemma for rectangles:} 
Let $\Pi = [0,R]\times [-c,c]$ be a rectangle and let $\Pi_1\subset \Pi_2\subset \Pi$ be two smaller rectangles of the
form $\Pi_i = [0,R]\times [-c_i, c_i]$ ($i=1,2$), $0< c_1<c_2<c$.  Let $\phi: \Pi_2\to \Pi$ be an area-preserving embedding (we
assume $\Pi$ is equipped with some area form) such that
\begin{itemize}

\item{} $\phi$ is the identity near $0\times [-c_2, c_2]$ and $R\times
[-c_2, c_2]$.

\item{} The area in $\Pi$ bounded by the curve $[0,R]\times y$ and
its image under $\phi$ is zero for some (and hence for all) $y\in
[-c_2, c_2]$.
\end{itemize}

\noindent If  $\phi$ is sufficiently $C^0$-small, then there exists $\psi\in Ham (\Pi)$ such that
$$\left. \psi \right|_{\Pi_1} = \phi\;\;\;\; {\it and} \;\;\;\; d_{C^0}(Id, \psi) \leq (d_{C^0}(Id, \phi))^{\frac{1}{2}}. $$

\begin{center}
{\bf Proof of Proposition \ref{Fragmentation} }
\end{center}
Postponing the proofs of the above extension lemmas, we will now use them to prove Proposition \ref{Fragmentation}. We will be closely following the proof of $(*)$ in section 1.6.2 of \cite{EPP}. 
\begin{proof}(Proposition \ref{Fragmentation})

   {\bf Part (1):}  We will first establish the result in the case $\partial \Sigma \neq \emptyset$.  It will be proven by induction on the number of 1-handles $l$.  The base case $l=0$ is obvious.  Assume now that the lemma holds for any surface with boundary obtained from the disk by attaching $l$ $1$-handles and suppose $\Sigma_0$ denotes one such surface.  Let $\Sigma$ be a surface obtained from $\Sigma_0$ by attaching one $1$-handle.
   
  As in \cite{EPP}, we pick a diffeomorphism $\varphi : [-1,1]^{2} \to \overline{\Sigma-\Sigma_{0}}$, which is singular at the corners,
  and maps $[-1,1]\times \{-1,1\}$ into the boundary of $\Sigma_{0}$. Let $\Pi_{r}=\varphi([-1,1]\times [-r,r])$ and 
$\Sigma_{1}=\Sigma_{0}\cup \varphi([-1,1]\times \{s, \vert s \vert \geq \frac{1}{4}\})$.  Note that $\Sigma_{1}$ is obtained from the disk by attaching $l$ $1$-handles and hence there exists a neighborhood $\nu_1$ of the identity in $Ham(\Sigma_1)$ such that all $\phi \in \nu_1$ can be fragmented as described in the proposition.
  % Choose a diffeomorphism (singular at the corners) $\varphi : [-1,1]^{2} \to \overline{\Sigma-\Sigma_{0}}$, sending $[-1,1]\times \{-1,1\}$ into the boundary of $\Sigma_{0}$. Let $\Pi_{r}=\varphi([-1,1]\times [-r,r])$. Let 
%$\Sigma_{1}:=\Sigma_{0}\cup \varphi([0,1]\times \{s, \vert s \vert \geq \frac{1}{4}\})$.  Note that $\Sigma_{1}$ is obtained from the disk by attaching $l$ $1$-handles and hence there exists a neighborhood $\nu_1$ of the identity in $Ham(\Sigma_1)$ such that all $\phi \in \nu_1$ can be fragmented as described in the proposition.

   Suppose that $\phi \in Ham(\Sigma)$ and let $\epsilon = d_{C^0}(Id, \phi)$.  As in \cite{EPP}, 
 assuming that $\epsilon$ is small enough, we apply the extension lemma for rectangles to the chain of rectangles
$\Pi_{\frac{1}{2}}\subset \Pi_{\frac{3}{4}}\subset \Pi_{\frac{7}{8}}$ and to
the restriction of $\phi$ to $\Pi_{\frac{3}{4}}$; note that $\phi$ being Hamiltonian ensures that the hypothesis on 
the curve $[-1,1]\times \{y\}$ is satisfied.  We obtain a diffeomorphism $\phi_1 \in  Ham(\Pi_{\frac{7}{8}})$  which coincides with $\phi$ on $\Pi_{\frac{1}{2}}$. Hence, we can write
  $$\phi = \phi_1 h,$$
  where $h$ is supported in $\Sigma_1$.  The argument in \cite{EPP} shows that $h \in Ham(\Sigma_1)$.  Also, note that we get the following estimates using the inequality from the extension lemma for rectangles: 
  \begin{equation} \label{Estimate on norm of h}
  d_{C^0}(Id, \phi_1) \leq C_1 \epsilon^{\frac{1}{2}} \text{  and   } d_{C^0}(Id, h) \leq C_1 \epsilon^{\frac{1}{2}},
  \end{equation}
  for some constant $C_1$.  Now, by the induction assumption there exist $l+1$ disks $(D_i)_{2 \leq i \leq l+2}$ covering $\Sigma_1$ such that any $\psi \in \nu_1 \subset Ham(\Sigma_1)$ can be fragmented as described by the lemma.  For our surface $\Sigma$ we take the required $l+2$ disks to be $D_1 = \Pi_{\frac{7}{8}}, D_2, \cdots, D_{l+2}$.  It just remains to show that $h$ can be fragmented as required by the proposition: if $\epsilon$ is picked to be sufficiently small, then (\ref{Estimate on norm of h}) guarantees that $h \in \nu_1$.  Hence, $h = \phi_2  \cdots \phi_{l+2}$, where
   $\phi_i \in Ham(D_j) \text{ } (i,j \geq 2)$ and 
   $$d_{C^0}(Id, \phi_i) \leq C_2 \, (d_{C^0}(Id, h))^{2^{-l}} \leq C_3 \epsilon^{2^{-l-1}},$$ \noindent where $C_2, C_3$ are constants.  Note that the neighborhood $\nu$ associated to $\Sigma$ must be picked so that if $\phi \in \nu$, then first,
   the restriction of $\phi$ to $\Pi_{\frac{3}{4}}$ is small enough for the application of the extension lemma for rectangles, and second, the bound on the $C^0$-norm of $h$ from (\ref{Estimate on norm of h}) is small enough to ensure that $h \in \nu_1$.
   This finishes the proof of part (1) of Proposition \ref{Fragmentation}.
   
  \noindent {\bf Part (2):}   Consider a chain of small disks $D_1 \subset D_2\subset D$ embedded in $\Sigma$.  Let $\Sigma_1 = \overline{\Sigma \setminus D_0}$, where $D_0$ is a disk contained in the interior of $D_1$.  If $\phi \in Ham(\Sigma)$ is sufficiently $C^0$-small then we can apply the extension lemma for disks to the chain of disks $D_1 \subset D_2 \subset D$ and $\left. \phi\right|_{D_2}$, exactly as the extension lemma for rectangles was applied in the proof of part (1), and obtain two diffeomorphisms $\phi_1 \in Ham(D)$ and $h \in Ham(\Sigma_1)$ such that 
   \begin{itemize}
   \item $\phi = \phi_1 h$ 
   \item $d_{C^0}(Id, \phi_1) \leq C_1 (d_{C^0}(Id,\phi))^{\frac{1}{2}}$ and $d_{C^0}(Id, h) \leq C_1 (d_{C^0}(Id,\phi))^{\frac{1}{2}}$, for some constant $C_1$.
   \end{itemize}
   The argument from \cite{EPP} ensures that $h \in Ham(\Sigma_1)$.  Observe that $\Sigma_1$ is a surface with boundary which is obtained from the disk by the attachment of $2g$ $1$-handles, and hence if $\phi$ is sufficiently $C^0$-small the result follows by applying part (1) to $h$.
   %Note that the neighborhood of identity in $Ham(\Sigma)$, $\nu$, must be picked so that, if $\phi \in \nu$, then first, it is  small enough to make the application of the extension lemma for disks to the chain of disks $D_1 \subset D_2\subset D$ possible, and second, the diffeomorphism $h$ obtained from the application of the extension lemma remains small enough for the application of part (1).
\end{proof}
\begin{center}
{\bf Proofs of the extension lemmas}
\end{center}

The extension lemmas used to prove Proposition \ref{Fragmentation} follow easily, as described in section 1.6.3 of \cite{EPP}, from the following extension lemma for annuli.  This lemma is a modification of Lemma 4 in section 1.6.3 of \cite{EPP}.  The proof of this lemma contains most of the hard work that goes into proving Proposition \ref{Fragmentation}.  Once again, we have tried to follow the argument presented in \cite{EPP} as closely as possible.

\noindent 
{\bf Area-preserving extension lemma for annuli:}
{\it   Let $\mathbb{A} = S^1 \times [-3,3]$ be a closed annulus and let $\mathbb{A}_1 =
S^1\times [-1,1], \; \mathbb{A}_2= S^1\times [-2,2]$ be smaller annuli inside
$\mathbb{A}$. Let $\phi$ be an area-preserving embedding of a fixed open
neighborhood of $\mathbb{A}_1$ into $\mathbb{A}_2$ (we assume that $\mathbb{A}$ is
equipped with some area form $\omega$), so that for some $y\in
[-1,1]$ (and hence for all of them) the curves $S^1\times y$ and
$\phi(S^1\times y)$ are homotopic in $\mathbb{A}$ and the\ area\ in \ $\mathbb{A}$ \ bounded\
by \ $S^1\times y$\ and \ $\phi(S^1\times y)$\  is\ $0$.

\noindent If $\phi$ is sufficiently $C^0$-small, then there exists $\psi\in Ham(\mathbb{A})$ such that $\left.
\psi\right|_{\mathbb{A}_1} = \phi$ and $$d_{C^0}(Id, \psi) \leq C  (d_{C^0}(Id, \phi))^{\frac{1}{2}}$$
for some constant $C>0$.

Moreover, if for some arc $I\subset S^1$ we have that $\phi = Id$ outside a quadrilateral $I\times [-1,1]$ and $\phi (I\times
[-1,1])\subset I\times [-2,2]$, then $\psi$ can be chosen to be the identity outside $I\times [-3,3]$.
}

\noindent
{\bf Proof:}  We equip $\mathbb{A} = S^1 \times [-3,3]$ with the area form $\omega = dx \wedge dy$, where $x$ is the coordinate on $S^1$
and $y$ is the coordinate on $[-3,3]$.  Suppose that $d_{C^0}(Id, \phi) \leq \epsilon$.  Let  $Diff_{0,c} (\mathbb{A}_2)$ denote the connected component of identity in the group of compactly supported diffeomorphisms of $\mathbb{A}_2$. By Lemma 5 from section 1.6.3 of \cite{EPP} there exists $f \in Diff_{0,c} (\mathbb{A}_2)$ such that $d_{C^0}(Id, f) \leq C \epsilon$, $f = \phi$ 
on a neighborhood of $\mathbb{A}_1$.  Denote $\Omega = f^* \omega$.  Following the strategy in \cite{EPP}
we will find a diffeomorphism $h \in Diff_{0,c} (\mathbb{A}_2)$ with the following properties:
\begin{itemize}

\item{} $\left. h\right|_{\mathbb{A}_1} = Id$,

\item{} $h^* \Omega = \omega$,

\item{} $d_{C^0}(Id, h) \leq C' \; \epsilon^{\frac{1}{2}}$.

\end{itemize}

Note that the only requirement that is different than those in \cite{EPP} is the third one.  
Given such an $h$ the argument in \cite{EPP} implies the existence of $\psi$ with the 
required properties.  We will now describe the changes that must be made to the argument in 
\cite{EPP} to obtain $h$ with the above properties.
 
%Given such an $h$, we extend $fh$ by the identity to all of $\mathbb{A}$.  It can easily be seen that $fh$ is area preserving and its $d_{C^0}(Id, fh)$ satisfies the correct estimate.  However, to get a Hamiltonian diffeomorphism $fh$ must be modified in a way such that, for any $y \in [-3,3]$, the area between $S^1 \times \{y\}$ and its image under $fh$ is zero.  This can be done by perturbing $fh$ by an amount no larger than $C' \; \epsilon^{\frac{1}{2}}$ on $\mathbb{A}\setminus \mathbb{A}_2$.  We will now describe the changes that must be made to the argument in \cite{EPP} to obtain $h$ with the above properties.  

\medskip
\noindent 
{\bf 1. Preparations for the construction of $h$:}  In this section we change
$r = \epsilon ^{\frac{1}{4}}$ to $r = \epsilon ^{\frac{1}{2}}$.  Note that the requirement that 
$r > 3 \epsilon$ is satisfied if $\epsilon$ is picked to be small enough.  
The rest of this section needs no changes.

\medskip
\noindent 
{\bf 2. Adjusting $\Omega$ on $\Gamma$:}  This section requires no changes.  Our choice of $r$ does 
not affect this part.  In this section the authors obtain a diffeomorphism $h_3$, which they later arrange to satisfy  
 \begin{equation} \label{Estimate for h_3}
 d_{C^0}(Id, h_3) \leq \epsilon.
 \end{equation}

\medskip
\noindent 
{\bf 3. Adjusting the areas of the squares:}  First note that in this section the authors use
the fact that $\frac{\epsilon}{r} \to 0$ as $\epsilon \to 0$.  This fact remains true for us 
as well, since $\frac{\epsilon}{r} = {\epsilon}^{\frac{1}{2}}$.  Note that our choice of $r$ changes 
equation (1.6) to $$| t_i|\leq C_1 \frac{\epsilon}{r} = C_1 \epsilon^{\frac{1}{2}}.$$
Next, the authors pick nonnegative functions $\bar{\rho}_{i}$ supported in the interior of the squares $K_{i}$ so that
$\int_{K_{i}}\bar{\rho}_{i}\omega=r^{2}$ and $||\bar{\rho}_{i}||_{C^{0}}\leq C_{2}\epsilon^{-1/2}$.  Note that, because $\int_{K_{i}} \omega=r^{2}$ one can easily pick the functions $\bar{\rho}_{i}$ as above such that they satisfy the better estimate 
$$||\bar{\rho}_{i}||_{C^{0}}\leq C_{2}.$$

Define a function $\varrho$ on $\mathbb{A}$ by $\varrho := 1+\sum_{i=1}^N t_i \bar{\rho}_i.$  Note that for an appropriate choice 
of a constant $C_3$ we have
 \begin{equation}\label{Estimate on varrho}
||\varrho - 1||_{C^{0}}\leq C_3 \epsilon^{\frac{1}{2}} 
\end{equation}
The rest of this section of the proof is unaffected by our changes.

\medskip
\noindent {\bf 4. Finishing the construction of $h_+$: Moser's argument:}  The authors apply Moser's argument and 
obtain a diffeomorphism $h_4$ whose $C^0$-distance from the identity is bounded by the diameter of the squares $K_i$, which
 have side length $r= \epsilon^{\frac{1}{2}}$, hence for an appropriate choice of a constant $C_4$ we have:
 \begin{equation} \label{Estimate for h_4}
   d_{C^0}(Id, h_4) \leq C_4 \epsilon^{\frac{1}{2}}. 
 \end{equation}
 
 Finally, the authors obtain another diffeomorphism $h_5$ by applying Moser's argument to the forms $\omega$ and $\varrho \omega$.  Part (iii) of Moser's trick, which we proved above, and estimate (\ref{Estimate on varrho}) imply that
  \begin{equation} \label{Estimate for h_5}
   d_{C^0}(Id, h_5) \leq C_5 \epsilon^{\frac{1}{2}}. 
 \end{equation}
 
 Then, as in \cite{EPP}, we set $h+ = h_3 h_4 h_5$.  Estimates (\ref{Estimate for h_3}), (\ref{Estimate for h_4}), and (\ref{Estimate for h_5}) imply that $d_{C^0}(Id, h_+) \leq C_6 \epsilon^{\frac{1}{2}},$ which is what we needed.
 
 \medskip
\noindent {\bf 5. Final observation:}  This section is unaffected by our changes.

This finishes the proof of the modified version of the extension lemma for annuli.

\section{{\bf Applications to the theory of Calabi quasimorphisms}} \label{Application to the theory of Calabi quasimorphisms}

One can associate to each open subset, $U$, of a symplectic manifold a subgroup of $\widetilde{Ham}(M)$, the universal cover of $Ham(M)$.  This subgroup is defined by:
  $$\widetilde{Ham}_U:= \{ \widetilde{\phi^t_H}: Supp(H) \subset U\}.$$
  Similarly, we define $Ham_U:= \{\phi^1_H: \widetilde{\phi^t_H} \in \widetilde{Ham}_U \}.$
    $\widetilde{Ham}_U$ admits a homomorphism, $\widetilde{Cal}_U : \widetilde{Ham}_U \rightarrow \mathbb{R}$, called the Calabi homomorphism \cite{C}, \cite{B} defined as follows:  
        $$\widetilde{Cal}_U(\widetilde{\phi^t_H}):= \int_{0}^{1} \int_{M} H(t, \cdot) \omega^n dt.$$
   If the symplectic form $\omega$ is exact on $U$ then the above formula gives a well defined homomorphism, $Cal_U: Ham_U \rightarrow \mathbb{R}$, which is also called the Calabi homomorphism.
        
   If $U \subset V$ are open sets then $\widetilde{Cal}_V = \widetilde{Cal}_U$ on $\widetilde{Ham}_U$, and if $\omega$ happens to be exact on $U$ and $V$ then $Cal_V = Cal_U$ on $Ham_U$.  One may wonder if it is possible to coherently glue these Calabi homomorphisms together to form a map on the entire symplectic manifold.  It is well known that $Ham(M, \omega)$ is simple, and $\widetilde{Ham}(M, \omega)$ is perfect, see \cite{B}, and hence these groups admit no nontrivial homomorphisms to the real line.  However, it was first shown by Entov and Polterovich in \cite{EP1} that, under certain restrictions on $QH^*(M)$, $\widetilde{Ham}(M, \omega)$ admits a homogeneous quasimorphism which, in a sense, extends the mentioned Calabi homomorphisms.  We will briefly review their work here, and present two applications of Theorem \ref{Lipschitz Estimate} to their theory.  The interested reader is referred to \cite{EP1, EP2, EP3, EP4,Mc, U2} for further information on this subject.  \\
   
   A quasimorphism on a group G is a map $\mu : G \rightarrow \mathbb{R}$ which is a homomorphism up to a bounded error, i.e., there exists a constant $C > 0$ such that for all $\phi, \psi \in G$  
               $$ | \mu(\phi \psi) - \mu(\phi) - \mu(\psi) | \leq C.$$
               
  We say $\mu$ is homogeneous if $\mu(\phi^m) = m \mu(\phi)$, for all $m \in \mathbb{Z}$.  
  
   Let $e$ denote an idempotent in the quantum cohomology ring of $M$, i.e., $e*e = e$.  Given a Hamiltonian path $\phi^t_H$, $0 \leq t \leq 1$, where $H$ is taken to be the unique normalized Hamiltonian generating $\phi^t_H$, we define $\rho_e(\phi^t_H) := \rho(H;e)$.  The homotopy invariance property of spectral invariants implies that $\rho_e: \widetilde{Ham}(M, \omega)\rightarrow \mathbb{R}$ is a well defined map.  If there exists a constant $R$ such that $\forall H \in C^\infty([0,1] \times M)$ 
   \begin{equation}\label{Spectral Boundedness}
     \rho(H;e) + \rho(\bar{H};e) \leq R 
   \end{equation}
  
   \noindent then the map $\rho_e$ defines a quasimorphism on $\widetilde{Ham}$, see \cite{U2}.  It has been shown that such an idempotent exists in the quantum cohomology ring of many symplectic manifolds, e.g., the identity element $1 \in QH^*(\mathbb{C}P^n)$, where $\mathbb{C}P^n$ is equipped with the Fubini-Study symplectic structure. However, $\rho_e$ is not homogeneous, so we homogenize it by defining 
   $\mu : \widetilde{Ham}(M, \omega) \rightarrow \mathbb{R}$ by the formula:
   \begin{equation} \label{Spectral Calabi Quasimorphism}
      \mu(\widetilde{\phi^t_H}) =  vol(M) \lim_{m \to \infty} \frac{\rho_e(\widetilde{\phi^t_H}^m)}{m}.
   \end{equation}
   
   If the idempotent $e$ satisfies Equation (\ref{Spectral Boundedness}), then $\mu$ is a homogeneous quasimorphism (see \cite{U2}) which satisfies the so called Calabi property:  if $U$ is a displaceable open set then $\mu|_{\widetilde{Ham}_U} = \widetilde{Cal}_U.$  We will refer to the quasimorphism $\mu$ obtained via Equation (\ref{Spectral Calabi Quasimorphism}) as the spectral Calabi quasimorphism. 
   
   \subsection{{\bf A triangle like inequality for Calabi quasimorphisms}}
   We will need the following lemma for our applications:
   
   \begin{lem}\label{Modified  Triangle Inequality}
     Let $\mu: \widetilde{Ham} \rightarrow \mathbb{R}$ denote the spectral Calabi quasimorphism obtained from $\rho_e$ via Equation (\ref{Spectral Calabi Quasimorphism}).  Suppose $\phi^t, \psi^t \in PHam$.  $\mu$ satisfies the following triangle like inequalities:
    \begin{enumerate}
    \item $\mu(\phi^t \psi^t) \leq \mu(\phi^t) + vol(M) \rho_e(\psi^t)$
    \item $\mu(\phi^t \psi^t) \leq vol(M) \rho_e(\phi^t) + \mu(\psi^t)$.
    \end{enumerate}
  \end{lem}
  \begin{rem} \label{Remark to Modified Triangle Inequality}
  
   Note that the above lemma implies that $\mu(\phi^t) - \mu(\psi^t) \leq vol(M) \rho_e(\psi^{-t} \phi^t )$.
   One interesting consequence of this inequality is the fact that the spectral Calabi quasimorphism is Lipschitz continuous with respect to the Hofer metric.  This important fact was established in \cite{EP1} by somewhat different methods.
   \end{rem}
   
   \begin{proof} (Lemma \ref{Modified Triangle Inequality})
      We provide a proof for the first of the two inequalities and leave the second to the reader.
      Note that for any integer $m$ we have 
      $$(\phi^t \psi^t)^m = (\phi^t)^m \prod_{i=1}^{m}{(\phi^t)^{i-m} \psi^t (\phi^t)^{m-i} } .$$
      The triangle inequality of spectral invariants implies that:
      \begin{equation}\label{Consequence of the product trick}
      \rho_e((\phi^t \psi^t)^m )\leq \rho_e((\phi^t)^m) + \sum_{i=1}^{m}{\rho_e((\phi^t)^{i-m} \psi^t (\phi^t)^{m-i} )}.
      \end{equation}
      We claim that $\rho_e((\phi^t)^{i-m} \psi^t (\phi^t)^{m-i} ) = \rho_e({\psi^t})$.  To see this, first observe that for any $\theta^t \in PHam$ the path $(\theta^t)^{-1} \psi^t \theta^t$ is homotopic with fixed end points to the path $\theta^{-1} \psi^t \theta$.  Here is a homotopy from one path to the other:

$$\Lambda(s,t) = (\theta^{(1-t)s \,+\, t})^{-1} \psi^t \theta^{(1-t)s\, + \,t}.$$

  The homotopy invariance property of spectral invariants implies that $\rho_e((\theta^t)^{-1} \psi^t \theta^t) = \rho_e(\theta^{-1} \psi^t \theta)$, and the latter equals $\rho_e(\psi^t)$ by the symplectic invariance property.  This proves the claim.
  
      It follows from inequality (\ref{Consequence of the product trick}) and the above claim that 
      $$\rho_e((\phi^t \psi^t)^m )\leq \rho_e((\phi^t)^m) + m \rho_e(\psi^t).$$
      Multiplying both sides of the above inequality by vol(M), dividing by $m$ and taking the limit as $m  \to \infty$ yields the result.  
   
   \end{proof}
 
   \subsection{First Application}
   We are now ready for the first application of Theorem \ref{Lipschitz Estimate}.  In the following theorem, we assume $e\in QH^*(M, \Lambda)$ satisfies Equation (\ref{Spectral Boundedness}). Let $\mu$ denote the spectral Calabi quasimorphism obtained from Equation (\ref{Spectral Calabi Quasimorphism}).  Let $U$ denote a proper open subset of $M$ and define $\eta : \widetilde{Ham}_U \rightarrow \mathbb{R}$ by
   $$\eta = \mu - \widetilde{Cal}_U.$$
   
   \begin{thm} \label{Application to Calabi Quasimorphisms}
  Suppose that $\phi^t, \psi^t \in Ham_{U}$ for all $t\in[0,1]$.  There exist constants $C, \delta > 0 $ depending on $U$, such that if $d_{C^0}^{path}(\phi^t, \psi^t) \leq \delta$ then:
   
       $$|\eta(\phi^t) - \eta(\psi^t)| \leq C \, d_{C^0}^{path}(\phi^t, \psi^t).$$
   \end{thm}
   \begin{proof}
   We pick $\delta$ to be the same constant from Theorem \ref{Lipschitz Estimate}.  Let $F, G: [0,1] \times M \rightarrow \mathbb{R}$ denote the unique normalized Hamiltonians which generate the flows $\phi^t \text{, } \psi^t$, i.e., $\phi^t= \phi^t_F$ and $\psi^t = \phi^t_G$.  Note that $\eta(\phi^t) - \eta(\psi^t) = \eta(\phi^t_F) - \eta(\phi^t_G) = \mu(\phi^t_F) - \mu(\phi^t_G) - \widetilde{Cal}_U(\phi^{-t}_G \phi^t_F).$  From Lemma \ref{Modified Triangle Inequality} and Remark \ref{Remark to Modified Triangle Inequality} we get $$\mu(\phi^t_F) - \mu(\phi^t_G) \leq vol(M) \rho(\bar{G} \# F;e),$$
   \noindent which combined with the previous line gives us
    $$\eta(\phi^t) - \eta(\psi^t) \leq vol(M) \rho(\bar{G} \# F;e) - \widetilde{Cal}_U(\phi^{-t}_G \phi^t_F)$$
    $$= vol(M) \rho(\bar{G} \# F + \frac{ \widetilde{Cal}_U(\phi^{-t}_G \phi^t_F)}{vol(M)} ;e),$$
 where the last equality follows from Property (\ref{Shift}) of spectral invariants.  Now, observe that the Hamiltonian $\bar{G} \# F + \frac{ \widetilde{Cal}_U(\phi^{-t}_G \phi^t_F)}{vol(M)}$  is supported in $U$.  Hence, Theorem \ref{Lipschitz Estimate} and the above inequality imply that
            $$\eta(\phi^t) - \eta(\psi^t) \leq C \, d_{C^0}^{path}(\phi^t, \psi^t),$$
\noindent for an appropriately chosen constant $C$.  Similarly, we get an estimate for $\eta(\psi^t) - \eta(\phi^t)$ from which the result follows.
   \end{proof}

\subsection{Second Application}   
   Let $B^{2n}_r$ denote the open ball of radius $r$ in $\mathbb{R}^{2n}$, equipped with the standard symplectic form $\omega_{st}$.  Let $\mathcal{H}(B^{2n}_{r})$ denote the $C^0$-closure of $Ham(B^{2n}_r)$ inside compactly supported homeomorphisms of $B^{2n}_r$.  In \cite{EPP}, Entov, Polterovich, and Py construct an infinite dimensional family of homogeneous quasimorphisms on $\mathcal{H}(B^{2n})$.  We will now present a brief summary of their work.
   
   Consider $\mathbb{C}P^n$ equipped with the Fubini-Study symplectic form, $\omega_{FS}$, normalized so that the integral of this form over the projective line is $1$.  In \cite{BEP}, the authors construct embeddings, $\theta_{\delta}:B^{2n}_{r_0} \rightarrow \mathbb{C}P^n$, where $r_0 = \frac{1}{\sqrt{\pi}}$ and $\delta$ is a parameter ranging over $(0,1]$.  These embeddings are conformally symplectic: $\theta_{\delta}^* \omega_{FS} = \delta \omega_{st}$.  The embeddings $\theta_{\delta}$ induce monomorphisms
$\theta_{\delta,*} : Ham(B^{2n}_{r_0}, \omega_{st}) \rightarrow Ham(\mathbb{C}P^n, \omega_{FS})$.  The Hamiltonian diffeomorphisms that are in the image of $\theta_{\delta,*}$ are supported in the interior of the image of $\theta_{\delta}$ and are given by the formula:
    $$\phi \mapsto \theta_{\delta}\phi \theta^{-1}_{\delta}.$$
    
    Let $\mu: Ham(\mathbb{C}P^n, \omega_{FS}) \rightarrow \mathbb{R}$ denote the spectral Calabi quasimorphism obtained from homogenization of $\rho(\cdot;1)$ \cite{EP1}.  It is shown in \cite{BEP} that $\mu_{\delta} := \delta^{-n-1} \mu \circ \theta_{\delta,*}$ is a Calabi quasimorphism on $Ham(B^{2n}_{r_0}, \omega_{st})$.  %(The authors of \cite{BEP} also show that for $\delta$ sufficiently close to $1$ the quasimorphisms $\theta_{\delta}$ are linearly independent.)  
    
    In \cite{EPP} the authors consider homogeneous quasimorphisms $$\eta_{\delta} := \mu_{\delta} - Cal_{B^{2n}_{r_0}}.$$  They show that each $\eta_{\delta}$ is bounded in a $C^0$-neighborhood of the identity in $Ham(B^{2n}_{r_0})$.  Employing general properties of homogeneous quasimorphisms one can show that this boundedness implies that $\eta_{\delta}$ is continuous with respect to the $C^0$-topology (see \cite{Sh}), and it extends continuously to $\mathcal{H}(B^{2n}_{r_0})$.  Using the fact that $B^{2n}$ is conformally symplectomorphic to $B^{2n}_{r_0}$, one can easily transfer all of these construction to the ball of radius $1$.
    
    Below, we will improve the results in \cite{EPP} by obtaining estimates which establish firstly, local Lipschitz continuity of $\eta_{\delta}$ with respect to $d_{C^0}$ on $Ham(B^{2n}_{r_0})$, and secondly, extension of $\eta_{\delta}$ to a (locally Lipschitz) homogeneous quasimorphism on $\mathcal{H}(B^{2n}_{r_0})$.  Our proof is a direct corollary of Theorem \ref{Application to Calabi Quasimorphisms} and it does not appeal to the general properties of homogeneous quasimorphisms used in \cite{EPP}.
 
 \begin{thm} \label{Local Lipschitz Continuity of EPP's Quasimorphisms}
    There exist constants $C\text{, }\epsilon > 0$, depending on $\eta_{\delta}$, such that if  $d_{C^0}(\phi, \psi) \leq \epsilon$ then 
    
    $$|\eta_{\delta} (\phi) -  \eta_{\delta} (\psi) | \leq C \, d_{C^0}(\phi, \psi).$$
    
    Here $d_{C^0}$ is the distance induced by the standard metric on $\mathbb{R}^{2n}$, and $C$ is some constant depending on $\eta_{\delta}$.  
    Furthermore, $\eta_{\delta}$ extends to $\mathcal{H}(B^{2n}_{r_0})$, and the extension satisfies the same estimate as above.  
    \end{thm} 
   \begin{proof} (Theorem \ref{Local Lipschitz Continuity of EPP's Quasimorphisms})
     Note that because $\epsilon$ does not depend on $\phi$ or $\psi$ the estimate in the theorem proves more than local Lipschitz continuity of $\eta_{\delta}$.  In fact, the second assertion of the theorem about $\eta_{\delta}$ extending to $\mathcal{H}(B^{2n}_{r_0})$ follows easily from this estimate.  Hence, we will only provide a proof for the first assertion in the theorem.
    
    Let $U$ denote the image of $B^{2n}_{r_0}$ under the embedding $\theta_{\delta}$.  To avoid confusing the $C^0$-distance on $\mathbb{C}P^n$ and the one on  $B^{2n}_{r_0}$ we will use the notation $d_{C^0,\mathbb{C}P^n}$ to denote the distance associated to $\mathbb{C}P^n$, and use $d_{C^0}$ for $B^{2n}_{r_0}$.  We drop all tildes from our notation, because in this case both $\mu$ and $Cal_U$ descend from $\widetilde{Ham}_U$ to $Ham_U$.
    
    It is easy to show that the ratio of any two Riemannian metrics on a compact manifold is always bounded, below and above.  This fact implies that there exist constants $A_1$, $A_2$ such that
    \begin{equation} \label{Equivalence of Distances}
    A_1 \leq \frac{d_{C^0,\mathbb{C}P^n}(\theta_{\delta,*}(\phi), \theta_{\delta,*}(\psi))}{d_{C^0}(\phi, \psi)} \leq A_2,
    \end{equation}
    for any homeomorphisms $\phi$, $\psi$.
    
    Suppose that $H: [0,1] \times B^{2n}_{r_0} \rightarrow \mathbb{R}$ is a normalized Hamiltonian.  It can easily be checked that $\theta_{\delta,*}(\phi^t_H)$ is generated by the Hamiltonian $\delta H(t,\theta_{\delta}^{-1}(x))$.  For simplicity of notation we will let $H^* = \delta H(t,\theta_{\delta}^{-1}(x))$, for any Hamiltonian $H: [0,1] \times B^{2n}_{r_0} \rightarrow \mathbb{R}.$  A simple computation, which will be carried out at the end of this proof, yields the following formula for $\eta_{\delta}$:
    
    \begin{equation}  \label{Simplified Formula}       
    \eta_{\delta} (\phi^1_H) = \delta^{-n-1} ( \, \mu(\phi^1_{H^*})  - Cal_U( \phi^1_{H^*}) \,).
    \end{equation}
    
    We will now prove our Theorem using the above formula.  Pick a Hamiltonian $F: [0,1] \times B^{2n}_{r_0} \rightarrow \mathbb{R} $ such that $\phi = \phi^1_F$.   By Lemma \ref{Local Path Connectedness} there exists $K: [0,1] \times B^{2n}_{r_0} \rightarrow \mathbb{R}$ such that $\phi^{-1} \psi = \phi^1_K$ and $$d_{C^0}^{path}(Id, \phi^t_K) \leq d_{C^0}(Id, \phi^{-1} \psi) =   d_{C^0}(\phi, \psi).$$  Let $G = F \# K$, so that $\phi^1_G = \psi$.  Note that we have:
    $$ d_{C^0}^{path}(\phi^t_F, \phi^t_G)   \leq d_{C^0}(\phi, \psi),$$
    which, by (\ref{Equivalence of Distances}), gives us the following estimate:
    \begin{equation} \label{Small Path}
    d_{C^0,\mathbb{C}P^n}^{path}(\phi^t_{F^*}, \phi^t_{G^*})  \leq A_2 d_{C^0}(\phi, \psi). 
    \end{equation}
    
    Formula (\ref{Simplified Formula}) tells us that $\eta_{\delta}$ is nothing but $\delta^{-n-1}$ times the pull back to $B^{2n}_{r_0}$ of the quasimorphism $\eta$ considered in Theorem \ref{Application to Calabi Quasimorphisms}.  The result follows immediately from (\ref{Small Path}), and Theorem \ref{Application to Calabi Quasimorphisms}.  Note that we must pick $\epsilon$ to be small enough to make the application of Theorem \ref{Application to Calabi Quasimorphisms} possible.
   % If $\delta = 1$, then $\eta_1$ is the pull back to $B^{2n}_{r_0}$ of the quasimorphism $\eta$ considered in Theorem \ref{Application to Calabi Quasimorphisms}.  The result follows immediately from (\ref{Small Path}), and Theorem \ref{Application to Calabi Quasimorphisms}.  Note that we must pick $\epsilon$ to be small enough to make the application of Theorem \ref{Application to Calabi Quasimorphisms} possible.  In the case where $\delta \neq 1$ unraveling the definitions   reveals that $\eta_{\delta}$ is a multiple of the pull back to $B^{2n}_{r_0}$ of the quasimorphism $\eta$ considered in Theorem \ref{Application to Calabi Quasimorphisms}: see formula (\ref{Simplified Formula}) below.  Hence, the result follows from the same argument as above.  The rest of this proof is dedicated to deriving formula (\ref{Simplified Formula}).
    
     We will now give a proof of formula (\ref{Simplified Formula}).  Because $\theta_{\delta}^* \omega_{FS} = \delta \omega_{st}$ we have 
     $$Cal_{B^{2n}_{r_0}}(\phi^1_H) = \int_{0}^{1} \int_{B^{2n}_{r_0}} H(t,x) \,\omega_{st}^n \, dt $$
            $$= \int_{0}^{1} \int_{\mathbb{C}P^n} H(t,\theta_{\delta}^{-1}(x)) \,(\theta_{\delta}^{-1})^*\omega_{st}^n \, dt$$
            $$= \delta^{-n} \int_{0}^{1}  \int_{\mathbb{C}P^n} H(t,\theta_{\delta}^{-1}(x)) \, \omega_{FS}^n \, dt$$
            $$= \delta^{-n-1} \int_{0}^{1}  \int_{\mathbb{C}P^n} H^*(t,x) \, \omega_{FS}^n \, dt$$
            $$= \delta^{-n-1} Cal_U(\phi^1_{H^*}).$$
               
      Also, by definition of $\mu_{\delta}$ we have
     $$\mu_{\delta}(\phi^1_H) = \delta^{-n-1} \mu( \theta_{\delta,*}(\phi^1_H) ) = \delta^{-n-1} \mu(\phi^1_{H^*}).$$
           
     Combine the above two computations to get (\ref{Simplified Formula}).
    
    \end{proof}

\section{{\bf $C^0$ Symplectic Topology and Spectral Hamiltonian Paths}} \label{Application to Spectral Hamiltonian Homeomorphisms}

 Suppose $H \in C^{\infty}([0,1] \times M)$ with the associated flow $\phi^t_H \in PHam(M)$.  Recall that for each $s \in [0,1]$ the Hamiltonian diffeomorphism $\phi^s_H$ is the time-1 map of the flow of the Hamiltonian $$H^s(t,x) = sH(st,x).$$  
    We define $\rho_H : [0,1] \to \mathbb{R}$, the spectral wave front function of $H$, by
    $$\rho_H(s) = \rho(H^s;1).$$
    
    This definition first appeared in an unpublished manuscript of Y.-G. Oh.  Oh used the above notion to define a $C^0$ generalization of smooth Hamiltonian paths.  We will now recall Oh's construction of spectral Hamiltonian paths and answer a question raised by him on this subject.
    
     By an isotopy of $M$ we mean a path in the group of homeomorphisms of $M$.  We assume that all Hamiltonians are normalized in the sense that $\int_{M} H(t, \cdot) \omega^n  = 0$ for each $t \in [0,1]$.
    
    \begin{defn}(Oh)  \label{Spectral Hamiltonian Paths}
    Suppose that $\phi^t : M \rightarrow M \text{ } (0\leq t \leq 1)$ is an isotopy of $M$ such that there exist a sequence 
    $\phi^t_{H_i}$ in $PHam(M)$ and two continuous functions $\rho, \bar{\rho}:[0,1] \rightarrow \mathbb{R}$ with the following properties:
    $$ \lim_{i\to \infty} d_{C^0}^{path}(\phi^t, \phi^t_{H_i}) = 0 \text{, } (C^0)\lim_{i \to \infty} \rho_{H_i} = \rho \text{, and   } (C^0)\lim_{i \to \infty} \rho_{\bar{H_i}} = \bar{\rho}.$$
    
    We call such an isotopy $\phi^t$ a spectral Hamiltonian path with the spectral wavefront function $\rho$.  By $PHameo_{sp}(M, \omega)$  we denote the set of all spectral Hamiltonian paths.  We define the set of spectral Hamiltonian homeomorphisms of $M$ by
$$Hameo_{sp}(M, \omega) := \{ \phi^1 : \phi^t \in PHameo_{sp}(M, \omega) \}.$$
   
    \end{defn}
    We will eliminate the symplectic form $\omega$ from the notation, unless there is a possibility of confusion.
    Recall that if $\lim_{i\to \infty} d_{C^0}^{path}(\phi^t, \phi^t_{H_i}) =0$ then we also have $\lim_{i\to \infty} d_{C^0}^{path}((\phi^t)^{-1}, \phi^t_{\bar{H_i}}) =0$.  Thus, the above definition implies that $\bar{\rho}$ is the spectral wavefront function of $(\phi^t)^{-1}$.  We should point out that it is not known if $PHameo_{sp}$ and $Hameo_{sp}$ are groups.  The difficulty here lies in showing that these sets are closed under composition.
    
    In the above definition, it is assumed that $M$ is closed.  However, we will need the above notions in the case of one non-closed manifold: the two dimensional disk $D^2$.  We embed $D^2$ into the two sphere as the southern hemisphere and we assume that all diffeomorphisms, homeomorphisms, and isotopies considered have supports contained in the interior of the southern hemisphere of $S^2$.  $PHameo_{sp}(D^2)$ and $Hameo_{sp}(D^2)$ are then defined as in Definition \ref{Spectral Hamiltonian Paths}.  Note that, even though we require that all diffeomorphisms and Hamiltonian paths be supported in the interior of $D^2$, we continue to assume that all Hamiltonians are normalized as Hamiltonians on $S^2$ and hence they may be non-zero functions of time in the northern hemisphere.  
    
    The following result answers Question \ref{Oh's Question} in the case where $M = D^2$.  
   \begin{thm}\label{Answer to Oh's Question}
    $Hameo_{sp}(D^2) = Sympeo_0(D^2)$.
   \end{thm}
   \begin{proof}
   Recall that we are assuming that $D^2$ is embedded into $S^2$ as the southern hemisphere.
   Suppose $\phi \in Sympeo_0(D^2)$.  Take a path $\phi^t \text{ } (0 \leq t\leq 1)$ in $Sympeo_0(D^2)$ such that $\phi^0 = Id$ and $\phi^1 = \phi$.  To obtain the result we have to show that $\phi^t \in PHameo_{sp}(D^2)$.  
   
   There exist smooth Hamiltonians paths $\phi^t_{i} \in PHam(D^2)$ such that 
                  $$d_{C^0}^{path}(\phi^t, \phi^t_{i}) \to 0.$$
   The existence of this sequence follows from the fact that every area preserving homeomorphism can be approximated by smooth area preserving diffeomorphisms; see \cite{Oh5, Si, Mu}. 
   
     Let $F_i$ denote the unique Hamiltonian supported in $D^2$ which generates $\phi^t_i$, i.e., $\phi_i^t = \phi^t_{F_i}$.  Now pick time independent Hamiltonians $f_i$ supported in balls of diameter $\frac{1}{i}$ contained in $D^2$ such that 
   $$\int_{S^2}{f_i} \omega = - \int_0^1 \int_{S^2}{F_i(t, \cdot)\omega} \; dt.$$
   
   Let $H_i = F_i \# f_i$.  Note that $H_i$ is supported in $D^2$ and $\int_0^1 \int_{S^2}{H_i(t, \cdot)\omega} \; dt = 0$.  Furthermore, because support of $f_i$ is contained in a ball of radius $\frac{1}{i}$ we have 
   $$d_{C^0}^{path} (\phi^t_{F_i}, \phi^t_{H_i}) \leq \frac{1}{i},$$
   and thus $$d_{C^0}^{path}(\phi^t, \phi^t_{H_i}) \to 0.$$
   
   It remains to show that the sequences of spectral wavefront functions $\rho_{H_i}$ and $\rho_{\bar{H_i}}$ have $C^0$ limits.
   To do so we will show that these sequences are Cauchy.
   
   For any small $\delta >0$, we have $d_{C^0}^{path}(\phi^t_{H_i}, \phi^t_{H_j}) \leq \delta$ for large enough $i, j$.  Because the $H_i$ and $H_j$ vanish on the northern hemisphere of $S^2$ we can apply Theorem \ref{Lipschitz Estimate} and get
        $$| \rho_{H_i}(1)  - \rho_{H_j}(1) | = | \rho(H_i;1) - \rho(H_j;1) | \leq C \, d_{C^0}^{path}(\phi^t_{H_i}, \phi^t_{H_j}) \leq C \; \delta. $$
   
   Similarly, we get that $$| \rho_{H_i}(s)  - \rho_{H_j}(s) | \leq C \; \delta, $$
   for any $s\in [0,1].$
   
   This shows that the sequence $\rho_{H_i}$ is Cauchy.  The same reasoning as above yields that $\rho_{\bar{H_i}}$ is Cauchy.
   This finishes the proof.
   \end{proof}
   
   \subsection{Failure of Uniqueness for wavefront functions}
     
   As in the case of topological Hamiltonian paths (see \cite{MO} for a definition), uniqueness issues turn out to be quite interesting in the case of spectral Hamiltonian paths.  In the case of topological Hamiltonian paths, it has been shown that (see \cite{BS, V3}) if
   $$d_{C^0}^{path}(Id, \phi^t_{H_i}) \to 0 \text{, and
   if } \exists \; H \text{ such that } \Vert H - H_i \Vert_{L^{(1, \infty)}} \to 0,$$ 
   then $H=0$.  
   The following theorem demonstrates that in the case of spectral Hamiltonian paths uniqueness of wave front functions fails, spectacularly.
   
   \begin{thm}\label{Failure of Uniqueness}
   Let $g : [0,1] \rightarrow \mathbb{R}$ denote any continuous function such that $g(0) = 0$.  Then, on any closed symplectic manifold $M$ there exists a sequence $\phi^t_{H_i} \in PHam(M)$ such that $$\displaystyle \lim_{i \to \infty} d_{C^0}^{path}(\phi^t_{H_i}, Id) = 0 \text{, and }   (C^0) \lim_{i \to \infty} \rho_{H_i} = g.$$
   \end{thm}
   \begin{proof}
   First, we assume that $g$ is differentiable.
   Let $K$ be a smooth, time-independent Hamiltonian, supported inside a Darboux chart $(U,x,y)$ such that $\int_{M}K=1.$ 
   Let $F(t,x,y) = g'(t)K(x,y),$ where $g'(t)$ is the first derivative of $g$.  Let $F_i(t,x,y) = i^{2n}F(t, ix, iy),$ where $2n$ is the dimension of the manifold.  Note that $supp(F_i)$ shrinks to a point and thus $d_{C^0}^{path}(\phi^t_{F_i}, Id) \to  0$.   Also, note that $$\int_{M} F_i(t,.) \omega^n = g'(t).$$ 
   
   Corollary 3.3 in \cite{U1} states that $\rho(F_i;1) \leq e(Supp(F_i),M)$, where $e(Supp(F_i);M)$ is the displacement energy of support of $F_i$.  The above inequality combined with the fact that $0 \leq \rho(F_i;1) + \rho(\bar{F_i};1)$, implies $\left| \rho(F_i;1) \right| \leq e(supp(F_i),M)$. Similarly, one can show that 
   $$\left| \rho(\bar{F_i};1)\right| \leq e(supp(F_i),M) .$$ 
   
   Since the sets $supp(F_i)$ shrink to a point, $e(supp(F_i);M)$ converges to $0$.   Thus, $\left| \rho(F_i;1) \right| $ and $\left| \rho(\bar{F_i};1) \right|$ converge to zero.  The same reasoning as above also implies that the spectral wavefront functions $\rho_{F_i}(s)$ and $\rho_{\bar{F_i}}(s)$ converge to $0$ uniformly.
   
   Let $H_i$ be the Hamiltonian obtained by normalizing $F_i$, i.e., $H_i(t,.) = F_i(t,.) - g'(t)$.  Then, 
     $$(C^0)\lim_{i\to \infty} \rho_{H_i}(s) =  (C^0)\lim_{i\to \infty}\rho(sH_i(st,.);1)$$
     $$ = (C^0)\lim_{i\to \infty}\rho(sF_i(st,.) - sg'(st);1)$$  
     $$ =  (C^0)\lim_{i\to \infty}\rho(sF_i(st,.);1) + \int_{0}^{1}{sg'(st) dt}$$
     $$ = (C^0)\lim_{i\to \infty}\rho_{F_i}(s) + \int_{0}^{s}{g'(t) dt} = g(s) - g(0) = g(s).$$
     
     Similarly, we see that $$(C^0)\lim_{i\to \infty} \rho_{\bar{H_i}}(s) = -g(s).$$
     
     If $g$ is not differentiable, pick a sequence of differentiable functions $g_i$ such that $\Vert g - g_i \Vert_{C^0} \leq \frac{1}{2i}$.  By the above, we can find smooth Hamiltonians $H_i$ such that:
     $$d_{C^0}^{path}(\phi^t_{H_i}, Id) \leq \frac{1}{i}, \text{   } \Vert \rho_{H_i}(s) - g_i(s)\Vert_{C^0} \leq \frac{1}{2i}, \text{ and }
     \Vert \rho_{\bar{H_i}}(s) + g_i(s)\Vert_{C^0} \leq \frac{1}{2i}.$$
     
      We, therefore, conclude that $$\phi^t_{H_i} \xrightarrow{C^0} Id,\text{   }\rho_{H_i} \xrightarrow{C^0} g \text{,   and }\rho_{\bar{H_i}} \xrightarrow{C^0} -g.$$
   
   \end{proof}
   
   Despite the above failure of uniqueness, we will next show that, in the case of surfaces, this failure is not as bad as it looks on the surface.  It would be very interesting to see if this result, which implies the $C^0$-continuity of the spectral norm, holds on general symplectic manifolds.
   
   \begin{prop}
   Suppose that $\Sigma$ is a surface and that $\phi^t_{H_i}$ is a sequence in $PHam(\Sigma)$ which converges uniformly, in time and space, to the identity.  Then, the sum of the spectral wave front functions of $H_i$ and $\bar{H_i}$, $\rho_{H_i} + \rho_{\bar{H_i}}$, converges uniformly to zero.
   \end{prop}
    \begin{proof}
    Let $\gamma_{H_i}(s)= \rho_{H_i}(s) + \rho_{\bar{H_i}}(s)$.  Note that,
           $\gamma_{H_i}(s) = \gamma(H_i^s).$
           
    Recall that the flow of the Hamiltonian $H_i^s$, where $s$ is fixed, is the path $\phi^{st}_{H_i} \text{ } (0 \leq t \leq 1)$, which converges uniformly to the identity.  Therefore, by Theorem \ref{Continuity of Gamma}, for $i$ large enough we have
                  $$\gamma(H_i^s) \leq C \, (d_{C^0}^{path}(Id, \phi^t_{H_i^s}))^{2^{-2g-1}} \leq C \, (d_{C^0}^{path}(Id, \phi^t_{H_i}))^{2^{-2g-1}}. $$
    The result follows from the above inequality.

    \end{proof}

\end{document}